\documentclass[11pt]{amsart}

\usepackage{amsfonts}
\usepackage{amscd}
\usepackage{amssymb}

\allowdisplaybreaks

\usepackage{hyperref}

\date{\today}

\newtheorem{thm}{Theorem}[section]

\newtheorem{lem}[thm]{Lemma}

\newtheorem{prop}[thm]{Proposition}
\theoremstyle{definition}
\newtheorem{defn}[thm]{Definition}
\theoremstyle{remark}
\newtheorem{rem}[thm]{Remark}

\numberwithin{equation}{section}

\newcommand{\R}{\mathbb R}

\newcommand{\He}{\mathbb H}

\newcommand{\C}{{\mathbb C}}

\DeclareMathOperator{\arccosh}{arccosh}

\renewcommand{\Im}{\operatorname{Im}}

\usepackage{color}

\title[Holomorphic extensions of eigenfunctions]
{Holomorphic extensions of eigenfunctions\\ on $NA$ groups}

\author[L. Roncal and S. Thangavelu]{Luz Roncal \and Sundaram Thangavelu}

\address[L. Roncal]{BCAM - Basque Center for Applied Mathematics \\
48009 Bilbao, Spain and Ikerbasque, Basque Foundation for Science, 48011 Bilbao, Spain}
\email{lroncal@bcamath.org }

\address[S. Thangavelu]{Department of Mathematics\\
 Indian Institute of Science\\
560 012 Bangalore, India}
\email{veluma@iisc.ac.in}


\begin{document}
\maketitle

\begin{abstract} 
Let $ X = G/K $ be a rank one Riemannian symmetric space of noncompact type. In view of the Iwasawa decomposition $ G = NAK $ of the underlying semisimple Lie group, we can also view $ X $ as the solvable extension $ S = NA $ of the Iwasawa group $ N$ which is known to be a $H$-type group. In this work we study the holomorphic extendability of eigenfunctions of the Laplace-Beltrami operator $ \Delta_S$ on $ S $ to certain domains in the complexification of the nilpotent group $ N$.  We can also do the same for any $H$-type group $ N $ not necessarily an Iwasawa group. The results are accomplished by making use of the connection with  solutions of the  extension problem for the Laplacian or the sublaplacian on the corresponding $ N$.
 \end{abstract}

\tableofcontents
\section{Introduction} 

It is known that Poisson transform of a function on the unit circle in the complex plane delivers a harmonic function on the unit disk. Moreover, each eigenfunction of the Laplace-Beltrami operator of the Poincar\'e disk can be expressed as the Poisson integral of a hyperfunction on the unit circle. In more generality, let $ X  =G/K $ be a non-compact Riemannian symmetric space where $ G $ is a connected semisimple Lie group and $ K $ a maximal compact subgroup of $ G$. The \textit{Helgason conjecture} claims that each joint eigenfunction of the invariant differential operators on $X$ has a Poisson integral representation by a hyperfunction boundary value \cite{H}. Helgason proved the statement for the Poincar\'e disk, and in the general case the conjecture was proven in \cite{KKMOOT}.

 In \cite{KS} Kr\"otz and Schlichtkrull showed that eigenfunctions of the Laplace-Beltrami operator $ \Delta_X $ on $ X $ have  holomorphic extensions to an open $ G $  invariant neighbourhood of $ X $ inside the complexification $ X_\C.$ If $ G_\C $ and $ K_\C $ stand for the complexifications of $G $ and $ K $ respectively, then $ X_\C = G_\C/K_\C$ has a natural complex structure and contains $ X $ as a totally real submanifold. Inside $ X_\C $ there is  a particularly simple  $ G $ -invariant open neighbourhood $ \Xi$ of $X $ first proposed by Akhiezer and Gidikin \cite{AG} and studied quite a lot in recent years by Kr\"otz and Stanton \cite{KSt1,KSt2}, see also the work by Camporesi and Kr\"otz \cite{CK}. It is called the complex crown and in \cite{KS} the authors proved that every eigenfunction of $ \Delta_X $ has a natural holomorphic extension to $ \Xi$. 

In the simple case of rank one symmetric spaces of non-compact type, the Iwasawa decomposition $ G = NAK $ allows us to identify $ X = G/K $ with   the  solvable group $ S = NA$. In this decomposition, $ N $ turns out to be a $H$-type group and $ A $, being identified with $ \R^+ $, has a natural action on $ N $ as dilations. Then the semidirect product $ S = NA $ can be identified with the symmetric space $ X$. In this setting, eigenfunctions of $ \Delta_S $ are functions on $ N \times A $. If $ u(n, a)$, $n \in N$, $a \in A  $, is an eigenfunction, we are interested in the holomorphic extendability of $ n \rightarrow u(n,a) $, \textit{for $ a $ fixed}. We expect that for each $ a $ there exists an open set $ \Omega_a$ inside the complexification $ N_\C $ of $ N $ to which all eigenfunctions of $ \Delta_S $ holomorphically extend. 

In particular, for hyperbolic spaces, where $ N $ is either $ \R^n $ or the Heisenberg group $ \He^n$, we pose the following slightly different problem: we fix $ a $ and ask for holomorphic extension in the $ n $ variable. In this paper we will show that this turns out to be true when $ X $ is a real or complex hyperbolic space. The motivation to study this problem comes from the relation between the eigenfunctions and the solutions to the generalised harmonic extensions. More precisely, we will prove that eigenfunctions of the Laplace-Beltrami operator $\Delta_S$ can be expressed as a Poisson integral of functions in certain weighted $L^2$ space if and only if for all $a$ the eigenfunction holomorphically extends to an open set $\Omega_{a}$, satisfying certain mild growth condition, see Theorem \ref{eigenR} and Theorem \ref{charE}. We will exploit the connection of the eigenfunctions with the solutions to the corresponding extension problems in the real and complex hyperbolic spaces. Indeed, we will prove characterisations of solutions of such extension problems, see Theorem \ref{thm:solepR} and Theorem \ref{charH}, and we will transfer the results to the eigenfunction problem. The key tool in the proofs of Theorems  \ref{thm:solepR} and Theorem \ref{charH} will be the use of the Gutzmer's formulas available in each case (in the case of the real hyperbolic space we will employ it implicitly). Our general approach to tackle both settings will be strongly based on spectral methods. As a final comment, we notice that an analogue of Helgason conjecture in general $NA$ groups remains open, see the end of Section \ref{complex}.

The paper is mainly organised in three parts. In Section \ref{real} we show the results for the real hyperbolic space, which are contained in Theorem  \ref{thm:solepR} and Theorem \ref{eigenR}, and in Section \ref{complex} for the complex hyperbolic space, Theorem \ref{charH} and Theorem \ref{charE}. Each section is self-contained and can be read independently. As expected, the case of complex hyperbolic space is far more involved.  Finally, in Section \ref{extensionH} we take up the case of general $H$-type groups. Using partial Radon transform in the central variable, we reduce the problem to the case of Heisenberg groups and we deduce results on the general $ NA $ groups from the results of Section \ref{complex}.

\section{Eigenfunctions on the real hyperbolic space} 
\label{real}

In this section we study the holomorphic extendability of eigenfunctions of the Laplace-Beltrami operator on real hyperbolic spaces. We do this by using the connection with the solutions of the extension problem for the Laplacian on $ \R^n$.

\subsection{Holomorphic properties of solutions of the extension problem}
\label{holosol}

By the extension problem for the Laplacian $ \Delta $ on $ \R^n $ we mean the  initial value problem, for $s>0$,
\begin{equation*}  
\big( \Delta + \partial_\rho^2 + \frac{1-s}{\rho}\partial_\rho )u(x,\rho) = 0 , \quad  u(x,0) = f(x), \qquad x \in \R^n, \,\,\rho> 0.
\end{equation*}
This problem has received considerable attention in recent times, especially after the  work by Caffarelli and Silvestre \cite{CS}. We are interested in the case when the initial condition $ f $ belongs to $ L^2(\R^n) .$ In this case the solution  is given by $ u(x,\rho) = \rho^sf \ast \varphi_{s,\rho}(x) $ where $ \varphi_{s,\rho} $ is the generalised Poisson kernel 
\begin{equation*} 
\varphi_{s,\rho}(x) =  \pi^{-n/2} \frac{\Gamma(\frac{n+s}{2})}{|\Gamma(s)|} (\rho^2+|x|^2)^{-\frac{n+s}{2}}.
\end{equation*}
We observe that the generalised Poisson kernel  has a holomorphic extension 
\begin{equation}
\label{geneP}
\varphi_{s,\rho}(z) =  \pi^{-n/2} \frac{\Gamma(\frac{n+s}{2})}{|\Gamma(s)|} (\rho^2+z_1^2+z_2^2+...+z_n^2)^{-\frac{n+s}{2}}.
\end{equation}
to the tube domain 
\begin{equation}
\label{tubeR} 
\Omega_{\rho} = \{ z=x+iy \in \C^n: |y| < \rho \} .
\end{equation}
It then follows that the solution $ u(x,\rho) $ also has a holomorphic extension $u(z,\rho)$ to $ \Omega_\rho.$ 

Let us recall some definitions and properties of Bessel functions that will be needed below. Let 
$J_{\nu}(z)$ and $I_{\nu}(z)$ be respectively the Bessel function of first kind of order $\nu$ and the modified Bessel function of first kind of order $\nu$ given by the formulas (see \cite[Chapter 5, Section 5.3 and Section 5.7]{Lebedev})
\begin{equation*}
J_{\nu}(z)=\sum_{k=0}^{\infty}\frac{(-1)^k(z/2)^{\nu+2k}}{\Gamma(k+1)\Gamma(k+\nu+1)}, \qquad |z|<\infty, \quad |\operatorname{arg} z|<\pi
\end{equation*}
and 
\begin{equation}
\label{eq:Inu}
I_{\nu}(z)=\sum_{k=0}^{\infty}\frac{(z/2)^{\nu+2k}}{\Gamma(k+1)\Gamma(k+\nu+1)}, \qquad |z|<\infty, \quad |\operatorname{arg} z|<\pi
\end{equation}
and let $ K_{\nu} $ be the Macdonald's function of order $\nu$ defined by (see \cite[Chapter 5, Section 5.7]{Lebedev})
\begin{equation}
\label{eq:Knu}
K_{\nu}(z)=\frac{\pi}{2}\frac{I_{-\nu}(z)-I_{\nu}(z)}{\sin \nu\pi}, \qquad |\operatorname{arg} z|<\pi, \quad  \nu\neq 0,\pm1, \pm 2, \ldots
\end{equation}
and, for integral $\nu=n$, $K_n(z)=\lim_{\nu\to n}K_{\nu}(z)$, $n= 0,\pm1, \pm 2, \ldots$.
From \eqref{eq:Inu} and \eqref{eq:Knu} it is clear that there exist constants $C_1, C_2, c_1,c_2>0$ such that
\begin{equation}
\label{eq:aszero}
c_1z^{\nu}\le I_{\nu}(z)\le C_1z^{\nu}, \quad c_2z^{-\nu}\le K_{\nu}(z)\le C_2z^{-\nu}, \quad \text{ for } z\to 0^+.
\end{equation}
Moreover, it is well known (see \cite[Chapter 5, Section 5.11]{Lebedev}) that, for small $\delta$,
\begin{equation*}
I_{\nu}(z)=C e^z z^{-1/2}+ R_{\nu}(z),\qquad 
| R_{\nu}(z)|\le C_{\nu} e^zz^{-3/2}, \quad |\operatorname{arg} z|\le\pi-\delta
\end{equation*}
and 
\begin{equation}
\label{eq:asymptotics-infiniteK}
K_{\nu}(z)=C e^{-z} z^{-1/2}+ \widetilde{R}_{\nu}(z),\qquad 
| \widetilde{R}_{\nu}(z)|\le C_{\nu} e^{-z}z^{-3/2}, \quad |\operatorname{arg} z|\le\pi-\delta.
\end{equation}

\begin{defn}
For any $ s > 0$ consider the following weight function
\begin{equation}
\label{weigh}
 w_\rho(\xi): =  c_{n,s} (\rho |\xi|)^s ( K_{s/2}(\rho |\xi|))^2 \frac{I_{s+n/2-1}(2\rho|\xi|)}{(2\rho|\xi|)^{s+n/2-1}}.
 \end{equation}
We define $\mathcal{H}^s_\rho(\R^n) $ to be the Sobolev space of all tempered distributions $ f $ whose Fourier transforms $ \widehat{f} $ are functions for which
$$
 \| f\|_{\mathcal{H}^s_\rho}^2 :=  \int_{\R^n}  |\widehat{f}(\xi)|^2 w_\rho(\xi) \,d\xi < \infty.
 $$
 Here, 
$$ 
\widehat{f}(\xi) := \int_{\R^n}f(x)e^{-ix\cdot \xi}\,dx.
$$
 \end{defn}
 
Recalling the definition of $ w_\rho $ in \eqref{weigh} and using the  asymptotic properties of $ K_\nu $ and $ I_\nu $ we check that 
\begin{equation} 
\label{asymw}
c (1+\rho |\xi|)^{-(n+1)/2}  \leq w_\rho(\xi) \leq C (1+\rho |\xi|)^{-(n+1)/2}. 
\end{equation}
From this it is clear that $ f \in L^2(\R^n) $ if and only if $ f \in \mathcal{H}^s_\rho(\R^n) $ and satisfies the uniform estimates $ \|f\|_{\mathcal{H}^s_\rho(\R^n)} \leq C $ for all $ \rho >0$.

\begin{rem}
Actually we can show that for a tempered distribution $ f $ whose Fourier transform is a function, 
$$  
\int_{\R^n} |\widehat{f}(\xi)|^2 (1+|\xi|^2)^{-(n+1)/4} d\xi < \infty 
$$ 
if and only if $ f \in \mathcal{H}^s_\rho $ for some $ \rho > 0 $ (and hence for all $ \rho >0$). However, as observed above, $ \| f\|_{\mathcal{H}^s_\rho}$ is uniformly bounded precisely when $ f \in L^2(\R^n)$.
\end{rem}
 
\begin{defn}
\label{wBerg}
For any $ s > 0 $ we define the weighted Bergman space $\mathcal{B}_s(\Omega_\rho) $ consisting of all holomorphic functions $ F $ on $ \Omega_\rho $ for which 
\begin{equation*} 
\| F \|^2_{\mathcal{B}_s} :=  \rho^{-n} \int_{\Omega_\rho}  |F(x+iy, \rho)|^2 \Big(1-\frac{|y|^2}{\rho^2}\Big)_+^{s-1}\,dx\, dy < \infty .
\end{equation*}
\end{defn}

\begin{prop}
\label{holoms}
Let $ f \in \mathcal{H}^s_\rho(\R^n) $ and $u(x,\rho) = \rho^s f \ast \varphi_{s,\rho}(x)$ be the solution to the extension problem
\begin{equation*}  
\big( \Delta + \partial_\rho^2 + \frac{1-s}{\rho}\partial_\rho )u(x,\rho) = 0 , \quad  u(x,0) = f(x), \qquad x \in \R^n, \,\,\rho> 0.
\end{equation*}
Then the solution has a holomorphic extension $ u(z,\rho) $ to $\Omega_{\rho}$ and  belongs to the weighted Bergman space $\mathcal{B}_s(\Omega_\rho) $.
\end{prop}

\begin{proof}
We begin by noting that the solution can  also be written as (see \cite{BGS})
\begin{equation*}
u(x,\rho)  = (2\pi)^{-n} \frac{2^{1-s/2}}{\Gamma(s/2)} \int_{\R^n}  \widehat{f}(\xi) (\rho |\xi|)^{s/2} K_{s/2}(\rho |\xi|) e^{i x\cdot \xi} d\,\xi.
\end{equation*}
It  follows from the asymptotic properties of $ K_{s/2} $ (see \eqref{eq:aszero} and \eqref{eq:asymptotics-infiniteK})  that for each $ \rho $ the function $ u(x,\rho) $ extends to the  domain $ \Omega_{\rho}$  as a holomorphic function. Then  by the Plancherel theorem for the Fourier transform, we have
\begin{equation}
\label{planch}
\int_{\R^n} |u(x+iy, \rho)|^2 dx  = (2\pi)^{-n} \frac{2^{2-s}}{\Gamma(s/2)^2} \int_{\R^n}  |\widehat{f}(\xi)|^2 (\rho |\xi|)^s ( K_{s/2}(\rho |\xi|))^2 e^{2 y\cdot \xi} d\xi.
\end{equation}
Recall that the Fourier transform of $(1-|y|^2)^{\alpha-1}_+$ is  explicitly known:
$$  
\int_{\R^n} (1-|y|^2)^{\alpha-1}_+ e^{-i y\cdot \xi} dy = c_{n,\alpha} |\xi|^{-\alpha-n/2+1}J_{\alpha+n/2-1}(|\xi|)
$$
where $ c_{n,\alpha}:=(2\pi)^{n/2} \Gamma(\alpha) 2^{\alpha-1}$.
 As the Bessel functions extend to $ \C $ as entire functions, it follows that
 \begin{equation} 
\label{FTweight}
\rho^{-n} \int_{\R^n} (\rho^2-|y|^2)_+^{s-1} e^{2y\cdot \xi} dy = c_{n,s} \rho^{2(s-1)}  (2\rho |\xi|)^{-s-n/2+1} I_{s+n/2-1}(2\rho |\xi|).
\end{equation} 

Multiplying both sides of \eqref{planch} by $ \rho^{-n}\big(1-\frac{|y|^2}{\rho^2}\big)_+^{s-1},$ integrating  with respect to $ y $ and making use of \eqref{FTweight} we have
\begin{multline}
\label{integration} 
\rho^{-n}\int_{|y| < \rho} \int_{\R^n} |u(x+iy, \rho)|^2 \Big(1-\frac{|y|^2}{\rho^2}\Big)_+^{s-1} \,dx \,dy\\
= d_{n,s}  \int_{\R^n}  |\widehat{f}(\xi)|^2 (\rho |\xi|)^s ( K_{s/2}(\rho |\xi|))^2 \frac{I_{s+n/2-1}(2\rho|\xi|)}{(2\rho|\xi|)^{s+n/2-1}} d\xi 
\end{multline}
where $ d_{n,s} := c_{n,s} (2\pi)^{-n} \frac{2^{2-s}}{\Gamma(s/2)^2}.$ 

The above calculations show that when $ f \in \mathcal{H}^s_\rho(\R^n) $, the solution $ u(z,\rho) $ belongs to the weighted Bergman space $\mathcal{B}_s(\Omega_\rho) $.
\end{proof}

From the proof of Proposition \ref{holoms} it is clear that we have  equality of norms:
$ \| u_\rho \|^2_{\mathcal{B}_s} =  \| f \|_{\mathcal{H}^s_\rho}^2 $,
where $ u_\rho(z) := u(z,\rho).$ This means that the map $ T_\rho $ taking $ f $ into the holomorphic function $  u_\rho $ is a constant multiple of an isometry. 
\begin{prop}
\label{unitary}
The map $ T_\rho: \mathcal{H}^s _\rho(\R^n) \rightarrow \mathcal{B}_s(\Omega_\rho) $ is unitary.
\end{prop}
\begin{proof} Suppose we are given  that $ u(z,\rho) $  is holomorphic on $ \Omega_ \rho $ and satisfies 
$$  
\rho^{-n}  \int_{\R^n} \int_{\R^n} |u(x+iy, \rho)|^2 \Big(1-\frac{|y|^2}{\rho^2}\Big)_+^{s-1}\,dx dy \leq C(\rho).
$$
Then for almost all $ |y| < \rho $ the integral  $\int_{\R^n} |u(x+iy, \rho)|^2 dx $ is finite. Without loss of generality we can assume that $\int_{\R^n} |u(x, \rho)|^2 dx $ is finite. For  any $ \varphi \in L^2(\R^n) $ whose Fourier transform is compactly supported, the function 
$ \int_{\R^n} u(a+z,\rho) \overline{\varphi}(a) \,da $, which is holomorphic on $ \Omega_  \rho $, agrees with $ \int_{\R^n} e^{ib\cdot z} \widehat{u}(b,\rho) \overline{\widehat{\varphi}}(b) db $ on $ \R^n $ and hence  
$$
\int_{\R^n} u(x+iy,\rho) \overline{\varphi}(x) dx =  \int_{\R^n} e^{- y\cdot \xi} \widehat{u}(\xi,\rho) \overline{\widehat{\varphi}}(\xi) \,d\xi.
$$
Consequently, we have
$$ 
\sup_{\widehat{\varphi} \in C_0^\infty(\R^n)}  \Big|\int_{\R^n} e^{- y\cdot \xi} \widehat{u}(\xi,\rho) \overline{\widehat{\varphi}}(\xi) d\xi\Big|     = \sup_{\widehat{\varphi} \in C_0^\infty(\R^n)}\Big| \int_{\R^n} u(x+iy,\rho) \overline{\varphi}(x) \,dx\Big| 
$$
which simply means that
$$ 
\int_{\R^n} |u(x+iy,\rho)|^2  dx = \int_{\R^n} e^{- 2y\cdot \xi} |\widehat{u}(\xi,\rho) |^2 \,d\xi.
$$
Integrating both sides with respect to the measure $ \rho^{-n} \big(1-\frac{|y|^2}{\rho^2}\big)_+^{s-1} dy$ and arguing as in \eqref{integration}, i.e. by using \eqref{FTweight}, we get
\begin{equation}
\label{integ}
\rho^{-n} \int_{\Omega_\rho}  |u(x+iy,\rho)|^2 \Big(1-\frac{|y|^2}{\rho^2}\Big)_+^{s-1}\,dx \,dy =   c_{n,s} \int_{\R^n}  |\widehat{u}(\xi,\rho) |^2  \frac{I_{s+n/2-1}(2\rho|\xi|)} {(2\rho|\xi|)^{s+n/2-1}} \,d\xi < \infty.
\end{equation}
We will make use of this identity in proving the unitarity of  the map $ T_\rho: \mathcal{H}^s _\rho(\R^n) \rightarrow \mathcal{B}_s(\Omega_\rho)$. 

Since $ T_\rho $ is an isometry all we need is to show that it is surjective. As the image of $ \mathcal{H}^s_{\rho}(\R^n) $ under $ T_{\rho} $ is closed, we only need to show that it is dense in $ \mathcal{B}_s(\Omega_\rho).$ To see this, suppose $ F \in \mathcal{B}_s(\Omega_\rho)$
is such that 
$$  
\int_{\Omega_\rho} F(z) \overline{T_{\rho}f}(z)  \Big(1-\frac{|y|^2}{\rho^2}\Big)_+^{s-1}\,dx\, dy = 0 
$$
for all $ f \in \mathcal{H}^s_{\rho}$. Taking \eqref{integ} into account, if $ F(x) = g(x)$, $x \in \R^n $, then
$$ 
\rho^{-n} \int_{\Omega_\rho} |F(z)|^2  \Big(1-\frac{|y|^2}{\rho^2}\Big)_+^{s-1} \,dx\, dy =  c_{n,s}\int_{\R^n} |\widehat{g}(\xi)|^2 \frac{I_{s+n/2-1}(2\rho|\xi|)}{(2\rho|\xi|)^{s+n/2-1}} \,d\xi.
$$
Polarising this identity and applying it to the pair $ (F, T_{\rho}f) $ we see that
$$ \int_{\R^n} \widehat{g}(\xi) \overline{\widehat{f}}(\xi) (\rho |\xi|)^{s/2} K_{s/2}(\rho|\xi|)  \frac{I_{s+n/2-1}(2\rho|\xi|)}{(2\rho|\xi|)^{s+n/2-1}} d\xi = 0.$$
Consequently, $ g = 0 $  which in turn implies $ F = 0 $ (as $ F $ is holomorphic), proving the denseness of the image of $ \mathcal{H}^s_\rho(\R^n) $ under $ T_\rho $ in $ \mathcal{B}_s(\Omega_\rho).$
\end{proof}

\begin{rem}
The weighted Bergman space $ \mathcal{B}_s(\Omega_\rho) $ in Definition \ref{wBerg} is a reproducing kernel Hilbert space and hence possesses a kernel $ K^\rho(z,w) =K^\rho_w(z) $ which is holomorphic in $ z $ and anti-holomorphic in $ w$ such that $ K^\rho_w \in \mathcal{B}_s $, and for all $ F \in \mathcal{B}_s $ one has $ F(w) = (F, K^\rho_w)_{\mathcal{B}_s}$. The kernel $ K^\rho(z,w) $ can be calculated explicitly. Indeed, let $ g = T_\rho f $ be an element of $ \mathcal{B}_s $ where $ f \in \mathcal{H}^s_\rho.$ Then on the one hand
$$ 
(g, K^\rho_w)_{\mathcal{B}_s} = g(w)= T_\rho f(w) = \pi^{-n/2} \frac{\Gamma(\frac{n+s}{2})}{|\Gamma(s)|} \rho^s  \int_{\R^n} f(y) (\rho^2+(w-y)^2)^{-(n+s)/2} dy.
$$
On the other hand
$$  (g, K^\rho_w)_{\mathcal{B}_s} = (T_\rho f, K^\rho_w)_{\mathcal{B}_s} =( f, T_\rho^{-1} K^\rho_w)_{\mathcal{H}_s}.$$
Comparing the two expressions we see that 
$$ T_\rho^{-1}K^\rho_w(y) =  \pi^{-n/2} \frac{\Gamma(\frac{n+s}{2})}{|\Gamma(s)|}\rho^s (\rho^2+(w-y)^2)^{-(n+s)/2} $$
 and consequently
\begin{equation*}
K^\rho(z,w) = \pi^{-n} \frac{\Gamma(\frac{n+s}{2})^2}{\Gamma(s)^2} \rho^{2s} \int_{\R^n} (\rho^2+(\overline{w}-y)^2)^{-(n+s)/2}(\rho^2+(z-y)^2)^{-(n+s)/2} \,dy.
\end{equation*}
By taking Fourier transform we also have the formula
\begin{equation}
\label{formulaKrho} 
K^\rho(z,w) =   (2\pi)^{-n} \frac{2^{2-s}}{\Gamma(s/2)^2}   \int_{\R^n}  (\rho |\xi|)^s (K_{s/2}(\rho |\xi|))^2  e^{i(\overline{w}-z)\cdot \xi}\, d\xi.
\end{equation}
To see this it is enough to consider $ z $ and $ w $ real as the function is holomorphic in $z $ and anti-holomorphic in $ w$. Then it is of the form $ \varphi_{s,\rho} \ast \varphi_{s,\rho}(z-\overline{w},\rho) $ where $ \varphi_{s,\rho} $ is the generalised Poisson kernel in \eqref{geneP}. By Fourier inversion we get that the formula for $ K^\rho(z,w) $ as the Fourier transform of  the Poisson kernel is given by the Macdonald function.

As $ K^\rho(z,w) $ is the reproducing kernel, every element $ F \in \mathcal{B}_\rho $ satisfies the pointwise estimate
$  | F(z)| \leq \|F\|_{\mathcal{B}_s} K^\rho(z,z)^{1/2} $. Indeed, 
$$
 | F(z)|=|\langle F,K_z^{\rho}\rangle_{\mathcal{B}_s}|\le  \|K_z^{\rho}\|_{\mathcal{B}_s} \|F\|_{\mathcal{B}_s}=\langle K_z^{\rho},K_z^{\rho}\rangle^{1/2}_{\mathcal{B}_s} \|F\|_{\mathcal{B}_s}= \|F\|_{\mathcal{B}_s}K^{\rho}(z,z)^{1/2}
$$
and hence by the definition of the reproducing kernel \eqref{formulaKrho},
$$ 
|F(z)|^2 \leq (2\pi)^{-n} \frac{2^{2-s}}{\Gamma(s/2)^2} \|F\|_{\mathcal{B}_s}^2 \int_{\R^n} (\rho |\xi|)^s (K_{s/2}(\rho |\xi|))^2  e^{2 y\cdot \xi} \,d\xi.
$$ 
Integrating in polar coordinates we have
$$ 
|F(z)|^2 \leq C_{n,s} \|F\|_{\mathcal{B}_s}^2 \int_0^\infty  (\rho r)^s (K_{s/2}(\rho r))^2  \frac{I_{n/2-1}(2r|y|)}{(2r|y|)^{n/2-1}} r^{n-1}\, dr.
$$ 
We observe that there is an explicit formula for  the above integral: indeed,   \cite[p. 402, 2.16.45, 6]{Prudnikov2} gives
\begin{multline*}
 \int_0^\infty  (\rho r)^s (K_{s/2}(\rho r))^2  \frac{I_{n/2-1}(2r|y|)}{(2r|y|)^{n/2-1}} r^{n-1} \,dr\\
 =\frac{1}{2^{n/2+1}}\sqrt{\pi}\frac{\Gamma(n/2+s)\Gamma((n+2)/2)}{\Gamma((n+s+1)/2)}\rho^{-n} {}_2F_1\Big(\frac{n}{2}+s,\frac{n+s}{2},\frac{n+s+1}{2};\frac{|y|^2}{\rho^2}\Big), \quad |y|<\rho.
\end{multline*}

\end{rem}

We are now ready to state and prove the following theorem on solutions of the extension problem.
\begin{thm} 
\label{thm:solepR}
A solution  of the extension problem  is of the form $ u(x,\rho) =  f \ast \rho^{s} \varphi_{s,\rho}(x)$ for some  $ f \in L^2(\R^n) $ if and only if for each $ \rho > 0$, $u(\cdot,\rho) $
extends to $ \Omega_\rho $ as a holomorphic function, belongs to $ \mathcal{B}_s(\Omega_\rho) $ and satisfies the uniform estimate  $\| u(\cdot,\rho) \|_{\mathcal{B}_s} \leq C $ for all $ \rho >0 .$ 
\end{thm}

\begin{proof}  
Let us first assume that $ u(x,\rho) =  f \ast \rho^{s} \varphi_{s,\rho}(x)$ for some  $ f \in L^2(\R^n) $. We have already proved in Proposition \ref{unitary} that when $ f \in \mathcal{H}^s_\rho(\R^n) $ the solution $ u = T_\rho f $ belongs to $ \mathcal{B}_s(\Omega_\rho) $ and $ \| u_\rho \|^2_{\mathcal{B}_s} =  \| f \|_{\mathcal{H}^s_\rho}^2 \leq C \|f\|_2^2.$ This proves the direct part of the theorem.

We now take up the converse.  Let us now suppose that $ u(z, \rho) \in \mathcal{B}_s $ for each $ \rho > 0 $ be such that 
\begin{equation*}
\rho^{-n} \int_{\Omega_\rho} |u(z,\rho)|^2  \Big(1-\frac{|y|^2}{\rho^2}\Big)_+^{s-1}  \,dx\, dy \leq C.
\end{equation*} 
In view of \eqref{integ} we see that $ \widehat{u}(\xi, \rho) $ is a locally integrable function. From \eqref{asymw} it follows that the function $ v(\xi,\rho) $ defined by $ \widehat{u}(\xi,\rho) = v(\xi,\rho) (\rho |\xi|)^{s/2} K_{s/2}(\rho |\xi|) $
satisfies the uniform estimates
$$ 
\int_{\R^n} |v(\xi,\rho)|^2 (1+|\xi|)^{-(n+1)/2} d\xi \leq C  \int_{\R^n} |v(\xi,\rho)|^2 w_\rho(\xi)  d\xi   = C \int_{\R^n}  |\widehat{u}(\xi,\rho) |^2  \frac{I_{s+n/2-1}(2\rho|\xi|)} {(2\rho|\xi|)^{s+n/2-1}} \,d\xi \leq C 
$$
for all $ 0< \rho \leq 1.$ Consequently, there is a subsequence $ \rho_k $ tending to $ 0 $ and a function $ g \in L^2(\R^n, (1+|\xi|)^{-(n+1)/2} d\xi) $ such that $ v(\cdot,\rho_k) \rightarrow g $ as $ k \rightarrow \infty $ weakly in
$L^2(\R^n, (1+|\xi|)^{-(n+1)/2} d\xi) .$ If $ f $ is the tempered distribution defined by $ \widehat{f} = g $ it then follows that $ f \in \mathcal{H}^s_\rho(\R^n) $ for every $ \rho >0.$ Moreover, for any Schwartz class function $ \varphi$,
$$ 
\int_{\R^n}  u(x,\rho_k) \varphi(x) dx = \int_{\R^n}  \widehat{u}(\xi,\rho_k) \widehat{\varphi}(\xi)\, d\xi =  \int_{\R^n} v(\xi,\rho_k) (\rho_k |\xi|)^{s/2} K_{s/2}(\rho_k |\xi|) \widehat{\varphi}(\xi)\, d\xi 
$$
from which we conclude that $ u(\cdot, \rho_k) $ converges to $ f $ in the sense of distributions.

If we further assume that $ u(x,\rho) $ is a solution of the extension problem then for any Schwartz class function $ \varphi $
we have 
$$ 
\int_{\R^n} -\Delta u(x,\rho) \overline{ \varphi}(x) dx = \Big( \partial_\rho^2 +\frac{1-s}{\rho} \partial_\rho \Big)\int_{\R^n} u(x,\rho)\overline{ \varphi}(x)\, dx
$$
which by Plancherel theorem leads to the equation
$$
\int_{\R^n}  |\xi|^2  \widehat{u}(\xi,\rho) \overline{ \widehat{\varphi}}(\xi)\, d\xi = \big( \partial_\rho^2 +\frac{1-s}{\rho} \partial_\rho \big)\int_{\R^n} \widehat{u}(\xi,\rho)\overline{ \varphi}(\xi) \,d\xi.
$$
Thus we see that $ \widehat{u}(\xi,\rho) $ satisfies the equation
$$ 
|\xi|^2  \widehat{u}(\xi,\rho)  = \Big( \partial_\rho^2 +\frac{1-s}{\rho} \partial_\rho \Big) \widehat{u}(\xi,\rho).
$$
Any solution of the above equation can be written as a linear combination of two linearly independent solutions (see \cite[Lemma 2.3]{BGS}):
$$ 
\widehat{u}(\xi,\rho)  = a(\xi) (\rho |\xi|)^{s/2} K_{s/2}(\rho |\xi|) + b(\xi) (\rho |\xi|)^{s/2} I_{s/2}(\rho |\xi|).
$$
As we have proved that $ u(\cdot, \rho_k) $ converges to the  tempered distribution $ f $ as $ \rho_k $ goes to zero,  for any $ \varphi \in C_0^\infty(\R^n) $ we have
$ (\widehat{f}, \varphi)  = c  \int_{\R^n} a(\xi) \varphi(\xi) \,d\xi $ as $(\rho_k |\xi|)^{s/2} K_{s/2}(\rho_k |\xi|) $ goes to a constant and $ (\rho_k |\xi|)^{s/2} I_{s/2}(\rho_k |\xi|)$ goes to zero. Thus the distribution $ \widehat{f} $ is given by the function $ c\, a $ and so we have
$$ 
\int_{\R^n} \widehat{u}(\xi,\rho) \varphi(\xi)\, d\xi  =  \int_{\R^n} \big( \widehat{f}(\xi) (\rho |\xi|)^{s/2} K_{s/2}(\rho |\xi|) + b(\xi) (\rho |\xi|)^{s/2} I_{s/2}(\rho |\xi|)\big) \varphi(\xi) \,d\xi.
$$
This means $ u(x,\rho) = \rho^{s} f \ast \varphi_{s,\rho}(x) + u_1(x,\rho) $ where
$$ 
u_1(x,\rho) = (2\pi)^{-n} \int_{\R^n} e^{i x\cdot \xi} b(\xi) (\rho |\xi|)^{s/2} I_{s/2}(\rho |\xi|)\, d\xi.
$$
As $I_{s/2}(\rho |\xi|)$ has exponential growth, the convergence  of the integral defining $ u_1 $ imposes severe restrictions of $ b.$ Our aim is to show that $ b = 0 $ under the hypothesis on $ u.$
Using the hypothesis on $ u $ and the relation \eqref{integ} we conclude that $ u_1 \in \mathcal{B}_s(\Omega_\rho) $ and 
$$  
\int_{\R^n} |b(\xi)|^2   (\rho |\xi|)^{s} I_{s/2}(\rho |\xi|)^2 \frac{I_{s+n/2-1}(2\rho|\xi|)} {(2\rho|\xi|)^{s+n/2-1}} \,d\xi \leq C .
$$ 
In view of the exponential growth of $ I_{s/2}(\rho |\xi|) $ this is possible only if $ b =0$. Thus $ u(x,\rho) =\rho^{s} f \ast \varphi_{s,\rho}(x) $  and the hypothesis gives
$$  
\int_{\R^n} |\widehat{f}(\xi)|^2 w_\rho(\xi) \,d\xi \leq C.
$$
By taking limit as $\rho $ tends to zero we obtain $ f \in L^2(\R^n)$ and this completes the proof of the theorem.

\end{proof}

\subsection{Holomorphic extensions of eigenfunctions of $ \Delta_g$}

We now consider the real hyperbolic space $ \mathbf{H} = G/K $ where $ G = SO_e(n,1)$ and  $K = SO(n) $. Here $SO_e(n,1)$ is the identity component of the group $SO(n,1).$ Eigenfunctions of the Laplace-Beltrami operator on $ \mathbf{H}$ (more generally on non-compact Riemannian symmetric spaces) have been characterised in the literature as Poisson integrals.  In the case of $ \mathbf{H} $, the Iwasawa decomposition $ G = NAK $ is given by $ N = \R^n$, $K = SO(n) $ and $ A = \R_+.$ Thus we can identify $ \mathbf{H} $ with the upper half-space $ \R^{n+1}_+  = \R^n \times \R_+ $  equipped with the Riemannian metric $ g = \rho^{-2} (|dx|^2+d\rho^2 ).$ 

We denote by $ \Delta_{g}$ the  Laplace-Beltrami operator associated to this metric. As this operator does not behave well with conformal change of metrics, we replace this with the new operator $ L_g(w) = -\Delta_g w- \frac{n^2-1}{4} w
$,  which is conformally covariant. By letting $ g_0 = \rho^2 g $ and making use of the conformal covariant property of $ L_g $ we calculate (see e.g. \cite{CG}) 
$$ L_gw = \rho^{\frac{n+3}{2}} L_{g_0}(\rho^{-\frac{n-1}{2}}w) =   - \rho^{\frac{n-1}{2}} \rho^2(\Delta+\partial_\rho^2) (\rho^{-\frac{n-1}{2}}w).$$
A simple calculation shows that
$$ L_g  =  -\rho^2\big(\Delta w+\partial_\rho^2 \big)w +(n-1) \rho \partial_\rho w -\frac{n^2-1}{4} w.$$
As $-  \Delta_{g} = L_g + \frac{n^2-1}{4} $,  the eigenfunction equation
$$  -\Delta_{g} w = \gamma(n-\gamma)w ,\quad \gamma = (s+n)/2$$
 becomes
$$ 
L_g w  = -\rho^2( \Delta + \partial_\rho^2)w + (n-1) \rho \partial_\rho w = \big( \frac{n^2-s^2}{4} \big) w.
$$
Defining $ u = \rho^{-\frac{n-s}{2}} w $ we easily check that $ w $ satisfies the above equation if and only if $ u $ satisfies the equation 
\begin{equation*}
 \big(\Delta+\partial_\rho^2+ \frac{1-s}{\rho} \partial_\rho\big) u = 0.
 \end{equation*}
 This establishes the connection between  certain eigenfunctions of $ \Delta_g $ and solutions of the extension problem.

 For any $ \lambda \in \C $  which is not a pole of $\Gamma(\frac{n-i\lambda}{2})$  we consider the kernels 
 $$ 
 \mathcal{P}_\lambda(x,\rho) =  \pi^{-n/2} \frac{\Gamma(\frac{n-i\lambda}{2})}{\Gamma(-i\lambda)} \rho^{\frac{n-i\lambda}{2}}(\rho^2+|x|^2)^{-\frac{n-i\lambda}{2}},
 $$
 which play the role of the Poisson kernels when $ \mathbf{H} $ is identified with the group $ S = NA$, $N =\R^n$, $A = \R_+$. Using these kernels, we can define the Helgason Fourier transform and prove the basic results in harmonic analysis. For example, the spherical functions $ \varphi_\lambda(x,\rho) $ on $ \mathbf{H} $ are given by
$$ \varphi_\lambda(x,\rho)  = \int_{\R^n} \mathcal{P}_\lambda(x-u,\rho) \mathcal{P}_{-\lambda}(u,1) du.$$  From the above representation, it is easy to prove the  following property which is crucial in the study of Fourier transform on $ \mathbf{H}$
\begin{equation}
\label{property}
\varphi_{\lambda}((y,r)^{-1}(x,\rho)) = \int_{\R^n}  \mathcal{P}_\lambda(x-u,\rho) \mathcal{P}_{-\lambda}(y-u,r) du.
\end{equation}
Recalling that the group law on $ S = \R^n \times \R^+ $ is given by 
$ (x,r)(x',r') = (x+rx', rr')$, the identity \eqref{property} can be easily verified by making a change of variables in the integral defining $ \varphi_\lambda$.   Indeed,
\begin{multline*}
 \varphi_{\lambda}((y,r)^{-1}(x,\rho))= \varphi_{\lambda}((-r^{-1}y, r^{-1})(x,\rho))=\varphi_{\lambda}(-r^{-1}y+r^{-1}x, r^{-1}\rho)\\
 = \int_{\R^n}
 \mathcal{P}_\lambda((-y+x)r^{-1}-u,^{-1}r\rho) \mathcal{P}_{-\lambda}(u,1)\, du,
\end{multline*}
and perform the change of variables $u\mapsto (u-y)/r$ to conclude the proof of the claim.

Given a reasonable function $ f $ on $ \R^n $ we define its Poisson transform by the equation
$$ 
\mathcal{P}_\lambda f(x,\rho)  =  \int_{\R^n} \mathcal{P}_\lambda(x-y,\rho) f(y) \mathcal{P}_{-\lambda}(y,1) du.
$$
In view of the discussion in Subsection \ref{holosol} and due to the connection between eigenfunctions and solutions to the extension problem just described at the beginning of this subsection, the Poisson transform $ \mathcal{P}_\lambda f $ is an eigenfunction of the hyperbolic Laplacian: $ \Delta_g (\mathcal{P}_\lambda f) =  -\frac{1}{4}(n^2+\lambda^2) \mathcal{P}_\lambda f$. Recall the definition of the tube domain $\Omega_{\rho}$ in \eqref{tubeR} and the weighted Bergman space in Definition \ref{wBerg}.

We can now prove  the following result on eigenfunctions of the hyperbolic Laplacian $ \Delta_g$ as a corollary of Theorem \ref{thm:solepR}.

\begin{thm} 
\label{eigenR}
An eigenfunction $ w(x,\rho) $ of the Laplacian $ \Delta_g $ with eigenvalue $ -\frac{1}{4}(n^2-s^2) $  is the Poisson integral $ \mathcal{P}_{is}f $ with $ f \in L^2(\R^n, (1+|y|^2)^{-n+s}dy)  $ if and only if $ w(x,\rho) $ extends to $ \Omega_\rho $ as a holomorphic function, belongs to $ \mathcal{B}_s(\Omega_\rho) $ and satisfies the estimate  $\| w(\cdot,\rho) \|_{\mathcal{B}_\rho} \leq C \rho^{(n-s)/2} $ for all $ \rho >0.$

\end{thm}
\begin{proof}
If we let $ w(x,\rho) = \mathcal{P}_{is}f(x,\rho) $ it follows that $ u(x,\rho) = \rho^{-\frac{(n-s)}{2}} w(x,\rho) = \rho^{s} g \ast \varphi_{s,\rho}(x) $ where $ g(y) = f(y) \mathcal{P}_{-is}(y,1).$ Observe that $ g \in L^2(\R^n) $ if and only if $  f \in L^2(\R^n, (1+|y|^2)^{-n+s}dy)  .$ The theorem follows immediately from Theorem \ref{thm:solepR}.

\end{proof}

\section{ Eigenfunctions on complex hyperbolic spaces}
\label{complex}

In this section we consider eigenfunctions of the Laplace-Beltrami operator on the complex hyperbolic space $ X = G/K $ where $ G = SU(n+1,1) $ and $ K= SU(n)$. Here $SU(n)$ is the special unitary group, that is, the Lie group of $n \times n$ unitary matrices with determinant $1$. 
As explained in the Introduction, it is known that the eigenfunctions holomorphically extend to a domain called the crown domain in the complexification $X_{\C} $ of $ X.$ As in the case of real hyperbolic case we identify $ X $ with a solvable group $ S = NA $ where $ G=NAK $ is the Iwasawa decomposition.
And as in the case of real hyperbolic space,  we treat $ X $ as the solvable group $ NA $  and consider eigenfunctions $ w(n,a) $ written in the coordinates $ n \in N $ and $ a \in A.$ In this setting we can fix $ a $ and consider the holomorphic extension of $ w(\cdot,a) $ to a domain in the complexification of $ N.$ We are going to show in this context similar results to what we have proved in Section \ref{real} for the real hyperbolic space.

\subsection{Solvable extension of the Heisenberg group} 
\label{solvable}

In the case of the complex hyperbolic space, the Iwasawa decomposition of $ G = SU(n+1,1) $ is explicitly given by $ N = \He^n$, $K = SU(n) $ and $ A = \R_+.$ Here $ \He^n $ is the Heisenberg group $\He^n=\C^n\times \R$ equipped with the group law
\begin{equation}
\label{law}
(z,\xi)(z',\xi')=\big(z+z', \xi+\xi'+\frac12 \Im (z\cdot \bar{z'})\big),
\end{equation}
where $z,z'\in \C^n$ and $\xi, \xi' \in \R$. We will often use real coordinates: thus identifying $ \He^n $ with $ \R^{2n+1} $ and considering coordinates $ (x,u,\xi) $ we can write the group law as
$$
(x,u,\xi)(y,v,\eta) = \big(x+y,u+v,\xi+\eta+\frac{1}{2}(u\cdot y-v\cdot x)\big), 
$$
where $x,u,y,v\in \R^n$ and $\xi,\eta \in \R$. 
Note that  $ \Im \big((x+iu)\cdot(y-iv)\big)= u\cdot y-v\cdot x = [(x,u)(y,v)] $ is the symplectic form on $ \R^{2n}.$

Let us recall that the convolution of $ f $ with $ g $ on $\He^n$ is defined by
$$
f*g(x) = \int_{\He^n} f(xy^{-1})g(y) \,dy, \quad x,y\in \He^n.
$$
 With  $x=(z,\xi)$ and $y=(z',\xi')$ the above takes the form
$$
f*g(z,\xi) = \int_{\He^n} f\big((z,\xi)(-z',-\xi')\big)g(z',\xi') \,dz' \,d\xi'.
$$
Let $ f^\lambda $ stand for the inverse Fourier transform of $ f $ in the \textit{last variable} $\xi$
\begin{equation}
\label{eq:inverseFT}
f^\lambda(z) = \int_{-\infty}^\infty f(z,\xi) e^{i\lambda \xi} \,d\xi.
\end{equation}
A simple computation shows that
\begin{equation}
\label{twist}
(f*g)^\lambda(z) = \int_{\C^n} f^\lambda(z-z')g^\lambda(z') e^{\frac{i}{2}\Im(z\cdot \bar{z'})} dz'.
\end{equation}
The convolution appearing on the right hand side is called the $ \lambda$-twisted convolution and is denoted by $ f^\lambda*_\lambda g^\lambda(z)$.

Observe that the center of this group is given by $ Z = \{ 0\} \times \R.$ Along with $ Z ,$ the following one parameter subgroups
$$ 
\Gamma_j = \{ (t e_j,0): t \in \R \},\qquad \Gamma_{n+j} = \{ (it e_j,0): t \in \R \}, \qquad j =1,2,..,n
$$ give rise to $ (2n+1) $ left invariant vector fields in the usual way. These are explicitly given by
\begin{equation}
\label{vfields}
X_j = \frac{\partial}{\partial{x_j}}+\frac{1}{2}u_j \frac{\partial}{\partial \xi},\qquad Y_j = \frac{\partial}{\partial{u_j}}-\frac{1}{2}x_j \frac{\partial}{\partial \xi},\qquad  T = \frac{\partial}{\partial \xi}.
\end{equation}
These vector fields form a basis for the Heisenberg Lie algebra $ \mathfrak{h}_n.$ It is easily checked that the only non-trivial Lie brackets in $ \mathfrak{h}_n $ are given by $ [ X_j, Y_j] = T$ as all other brackets vanish.  
The second order operator $ \mathcal{L} = -\sum_{j=1}^n ( X_j^2 +Y_j^2) $, known as the sublaplacian, plays the role of  Laplacian $ \Delta $ for the group $ \He^n$. Though not elliptic this operator shares several properties with its counterpart $ \Delta $ on $ \R^n$.

The group $ \He^n $  admits a family of automorphisms indexed by $ \R_+ $ and given by the non-isotropic dilations $ \delta_r (z,a) = ( rz, r^2 a).$ With respect to these dilations, the vector fields $ X_j, Y_j, T $ are homogeneous of degree one and $ \mathcal{L}$ is homogeneous of degree 2. Recall that a fundamental solution for $ \Delta $ on $ \R^n $ is given by a constant multiple of $ |x|^{-n+2}$, $n \neq 2 $. 
In the same way, a fundamental solution for $ \mathcal{L} $ is given by a constant multiple of $ |(z,\xi)|^{-Q+2} $ where $ Q = 2n+2$ and $ |(z,\xi)|^4 = |z|^4+16 \xi^2.$ The function $ (z,\xi) \rightarrow |(z,\xi)| $ is known as the Koranyi norm, which is homogeneous of degree one with respect to the dilations $ \delta_r$.

The quantity $ Q = 2n+2 $ is known as the homogeneous dimension for the following reason. The group $ \He^n $ turns out to be unimodular and the Haar measure is simply given by the Lebesgue measure $ \,dz \,d\xi $ on $ \C^n \times \R$. For $ f \in L^1(\He^n) $ we have
$$  
\int_{\He^n} f(\delta_r(z,\xi))\, dz\, d\xi =  r^{-Q} \int_{\He^n} f(z,\xi)\, dz\, d\xi .
$$ 
Moreover, on the Koranyi sphere $ K_1 = \{ (z,\xi) : |(z,\xi)| = 1 \} $ there exists a measure $ d\sigma $ so that the Haar measure on $  \He^n $ has the polar decomposition
$$ 
\int_{\He^n} f(z,\xi) \,dz \,d\xi = \int_0^\infty \Big( \int_{K_1} f(\delta_r(z,\xi)) d\sigma(z,\xi) \Big)   r^{Q-1} \,dr.
$$ 
By defining $ \sigma_r $ on the sphere $ K_r = \{ (z,\xi) : |(z,\xi)| = r \} $  by the prescription
$$ 
\int_{K_r} f(z,\xi) d\sigma_r = \int_{K_1} f(\delta_r(z,\xi)) d\sigma 
$$ 
we can write the polar decomposition in the form
$$ 
\int_{\He^n} f(z,\xi)\, dz\, d\xi = \int_0^\infty \Big( \int_{K_r} f(z,\xi)  d\sigma_r(z,\xi) \Big)   r^{Q-1} dr.
$$

As $ \R_+ $ acts on $ \He^n $ as automorphisms, we can form the semi-direct product $ S = \He^n \times \R_+$. The group law in $ S $ is given by
$$ 
(z,\xi,\rho) (w,\eta,\rho') = ((z,\xi)\delta_{\sqrt{\rho}}(w,\eta), \rho \rho') = \big(z+\sqrt{\rho}w, \xi+ \rho \eta+\frac{1}{2} \sqrt{\rho}  \Im(z\cdot \bar{w}), \rho\rho'\big).
$$
This group turns out to be a solvable group which is non-unimodular. Indeed, the left Haar measure on $ S $ is $ \rho^{-n-2}\, dz\, d\xi\, d\rho $ whereas the right Haar measure is $ \rho^{-1}\, dz\, d\xi\, d\rho.$ 
The Lie algebra $ \mathfrak{s} $ of the Lie group $ S $ can be identified with $ \R^{2n+1} \times \R.$  An easy calculation shows that the vector fields $ E_0 = \rho \partial_\rho$, $ E_j = \sqrt{\rho}X_j$, $E_{n+j} = \sqrt{\rho}Y_j $ for $ j = 1,2,...,n $ and $ E_{2n+1} = \rho T $ are left invariant.  The non zero Lie brackets are given by
$$  
[E_0, E_j] = \frac{1}{2} E_j, \qquad [E_0, E_{n+j}] = \frac{1}{2} E_{n+j}, \qquad [E_0, E_{2n+1}] = E_{2n+1}.
$$ 
Equipping $ \mathfrak{s} $ with the standard inner product on $ \R^{2n+2} $ these $2n+2 $ vector fields can be made to form an orthonormal basis for $ \mathfrak{s}. $ 
This induces a Riemannian metric on $ \mathfrak{s}.$

The Laplace-Beltrami operator on the Riemannian manifold $ S $ can be expressed in terms of the vector fields $ E_j.$ Indeed, we have
$$ 
\Delta_S = \sum_{j=0}^{2n+1} E_j^2-(n+1) E_0   = - \rho \mathcal{L}+\rho^2 \partial_{\xi}^2 +(\rho \partial_\rho)^2 -(n+1) \rho \partial_\rho.
$$
Recalling the expressions for the vector fields $ X_j $ and $ Y_j $ in \eqref{vfields} we have the more explicit formula
$$ 
\Delta_S = \rho \Big( \Delta_{\R^{2n}}+\frac{1}{4}(|x|^2+|u|^2) \partial_\xi^2-\sum_{j=1}^n \big(x_j \frac{\partial}{\partial u_j} - u_j \frac{\partial}{\partial x_j}\big) \partial_\xi \Big)+\rho^2 \partial_\xi^2 + (\rho \partial_\rho)^2 -(n+1) \rho\partial_\rho.
$$
As in the case of the hyperbolic Laplacian $ \Delta_g $ on $ \R^n \times \R_+$ we now consider the following eigenvalue problem 
$$ 
- \Delta_S \widetilde{W}(x,u,\xi,\rho) = \gamma(n+1-\gamma)\widetilde{W}(x,u,\xi,\rho),\quad \gamma=\frac{1}{2}(n+1+s),\quad s>0.
$$
We define $ W $ and $ U $ in terms of $ \widetilde{W} $ as follows:
$$  
\widetilde{W}(x,u,\xi,\rho)= \rho^{\frac{n+1-s}{2}}W(x,u,\xi,\rho) = \rho^{\frac{n+1-s}{2}}U(2^{-1/2}(x,u),2^{-1}\xi,\sqrt{ 2\rho}).
$$
An easy calculation establishes the following relation.

\begin{prop} 
\label{conne}
Let $s>0$. The function $ \widetilde{W} $ is an eigenfunction of the Laplace-Beltrami operator $ \Delta_S $ with eigenvalue $ -\gamma(n+1-\gamma)$, $\gamma=\frac{1}{2}(n+1+s)$ if and only if the function $ U $ defined as
$$
\widetilde{W}(x,u,\xi,\rho)=\rho^{\frac{n+1-s}{2}}U(2^{-1/2}(x,u),2^{-1}\xi,\sqrt{ 2\rho})
$$
 satisfies the equation
\begin{equation} 
\label{eph}
\big( -\mathcal{L} + \partial_\rho^2 +\frac{1-2s}{\rho} \partial_\rho +\frac{1}{4}\rho^2 \partial_\xi^2 \big) U(x,u,\xi,\rho) = 0.
\end{equation}
\end{prop}

We will make use of this connection between eigenfunctions of $ \Delta_S $ and solutions $ U $ of the equation~\eqref{eph}. The holomorphic properties of $ W $ (and $\widetilde{W}$) follow from that of $ U.$  By studying solutions of the equation \eqref{eph} with certain initial conditions, we will show that the solutions extend holomorphically to certain domains in the complexification of the Heisenberg group, which is naturally identified with $ \C^{2n+1}.$ Such a connection then allows us to interpret the results as properties of eigenfunctions of $ \Delta_S.$

\subsection{An extension problem for the sublaplacian on  $ \He^n$ and the Gutzmer formula}  

By the extension problem for the sublaplacian $ \mathcal{L}$ on $ \He^n $ we mean the following initial value problem, for $s>0$:
\begin{equation}
\label{epg}
\big( -\mathcal{L} + \partial_\rho^2 +\frac{1-2s}{\rho} \partial_\rho +\frac{1}{4}\rho^2 \partial_\xi^2 \big) U(x,u,\xi,\rho) = 0,\qquad  U(x,u,\xi,0) = f(z,\xi).
 \end{equation}
 As remarked earlier, the coordinates on $ \He^n $ will be denoted by $ (x,u,\xi).$ This change of notation is necessitated since we have to complexify the variables $x,u$ and $\xi.$
The extension problem \eqref{epg} has been studied extensively in the literature, see the works by Frank et al. \cite{FGMT}, M\"ollers et al. \cite{MOZ} and the authors \cite{RTimrn}. A solution of this problem is explicitly given by (see e.g. \cite[Theorem 1.2]{RTimrn})
 $$ 
 U(x,u,\xi,\rho) = \rho^{2s}  f \ast \Phi_{s,\rho}(x,u,\xi),\quad x, u \in \R^n,\quad \xi \in \R, \quad \rho>0 
 $$
 where the kernel $ \Phi_{s,\rho}$ is 
\begin{equation}
\label{Phi}
\Phi_{s,\rho}(x,u,\xi) = \frac{2^{n+1+s} }{ \pi^{n+1}\Gamma(s)}  \Gamma\Big(\frac{n+1+s}{2}\Big)^2 \big((\rho^2+|x|^2+|u|^2)^2+16\xi^2\big)^{- \frac{n+1+s}{2}} . 
\end{equation}
We make the observation that $ \Phi_{s,\rho}(x,u,\xi) $ has a holomorphic extension to $ \C^n \times \C^n \times \R $ as (below $C_{n,s}$ is the constant in \eqref{Phi})
$$  
\Phi_{s,\rho}(z,w,\zeta) =C_{n,s} \big( (\rho^2+z^2+w^2)^2+16 \zeta^2\big)^{-(n+1+s)/2},\quad z =x+iy,\quad w =u+iv, \quad \zeta = \xi+i\eta 
$$
where $ z^2 =(x+iy)^2 = \sum_{j=1}^n(x_j+iy_j)^2 $ (and analogously for $w$ and $\zeta$) provided 
$$ 
|y|^2+|v|^2+4|\eta| <\rho^2.
$$ 
Therefore, it is reasonable to expect that under suitable condition on $ f $, the function $U(x,u,\xi ,\rho)$
can be holomorphically extended to certain domain in $ \C^n \times \C^n \times \C$.  However, this domain has to be invariant under translations by elements of $ \He^n$. We will show that for $ f \in L^2(\He^n) $ the solution $ U$ defined above holomorphically extends to the domain $ \Omega_{\rho/4} $ where 
\begin{equation}
\label{tube}
\Omega_r = \Big\{ (z,w,\zeta) \in \C^{2n+1}: \big|\operatorname {Im}(z,w,\zeta-\frac{1}{4}(z\bar{w}-w\bar{z}))\big|< r \Big\}.
 \end{equation}

In order to proceed further, we need to recall the spectral decomposition of the sublaplacian $ \mathcal{L}.$  We begin by defining the scaled Laguerre functions of type $n-1$
$$
\varphi_{k}^{\lambda}(z)=L_k^{n-1}\Big(\frac12|\lambda||z|^2\Big)e^{-\frac14|\lambda||z|^2}.
$$
Here $L_k^{n-1}$ are the Laguerre polynomials of type $n-1$, see \cite[Chapter 1.4]{STH} for the definition and properties. 
Moreover, we use the notation 
\begin{equation*}
\varphi_k^{\lambda}(x,u):=L_k^{n-1}\Big(\frac12|\lambda|(x^2+u^2)\Big)e^{-\frac14|\lambda|(x^2+u^2)},
\end{equation*} 
and the definition above remains valid for $x,u\in \C^n$.
 Let $ d\mu(\lambda) = (2\pi)^{-n-1}|\lambda|^n \, d\lambda $, which plays the role of the Plancherel measure for $ \He^n.$ For any $ f \in L^2(\He^n) $ we have the decomposition
\begin{equation}
\label{decol}
f(x,u,\xi) =  \int_{-\infty}^\infty  e^{-i\lambda \xi} \big(\sum_{k=0}^\infty f^\lambda \ast_\lambda \varphi_k^\lambda(x,u)\big)  \, d\mu(\lambda) 
\end{equation}
where $ f^\lambda(x,u) $ stands for the inverse Fourier transform of $ f $ in the last variable as defined in \eqref{eq:inverseFT} and, by \eqref{twist},
$$ 
f^\lambda \ast_\lambda \varphi_k^\lambda(x,u) = \int_{\R^{2n}} f^\lambda(x-a,u-b) \varphi_k^\lambda(a,b) e^{\frac{i}{2}\lambda(u \cdot a-x \cdot b)} \,da\, db.
$$
In view of the fact that the functions $ e^{i\lambda \xi}\varphi_k^{\lambda}(x,u) $ are eigenfunctions of the sublaplacian with eigenvalues $ (2k+n)|\lambda| $, the expression in \eqref{decol} gives the spectral decomposition of $ f $ in terms of eigenfunctions of the sublaplacian.  The Plancherel theorem then reads as
$$ \int_{\He^n} |f(x,u,\xi)|^2 dx\,du\,d\xi = c_n  \int_{-\infty}^\infty  \big(\sum_{k=0}^\infty \| f^\lambda \ast_\lambda \varphi_k^\lambda\|_2^2 \big)  \,d\mu(\lambda) $$
for an explicit constant $ c_n.$ Under some assumptions on the decay of $ \| f^\lambda \ast_\lambda \varphi_k^\lambda\|_2 $ as a function of $ k $, the function $ f $ will extend to $ \C^n \times \C^n \times \C $ as a holomorphic function. It is then natural to ask for a formula for the $ L^2(\He^n) $ norm of the extended function. An answer to this question is provided by the so called Gutzmer's formula.

Let $G_n$ be the Heisenberg motion group, that is the semi-direct product $G_n=\He^n\ltimes U(n)$, where $U(n)$ is the group of $n\times n$ complex unitary matrices acting on $\He^n$ by the automorphisms
$$
\sigma\cdot (x,u,\xi)=(\sigma \cdot(x, u),\xi), \quad \sigma \in U(n).
$$
The Heisenberg motion group acts on $\He^n$ in the following way:
$$
(x,u,\xi,\sigma)(x',u',\xi')=\big((x,u)+\sigma\cdot(x',u'),\xi+\xi'+\frac12 \Im \sigma \cdot(x'+iu')\overline{(x+iu)})\big).
$$
This action has a natural extension to $ \C^n \times \C^n \times \R $. Let $dg $ stand for the Haar measure on $ G_n.$

\begin{thm}[\cite{TPac} Theorem 4.2] 
\label{thm:gutzmer}
Let $ F $ be an entire function on $ \C^{2n+1}$ and let $ f $ stand for the restriction of $ F $ to $ \He^n.$ Then we have the identity
$$ 
\int_{G_n} |F(g\cdot(z,w,\zeta))|^2 dg = c_n \int_{-\infty}^\infty e^{\lambda(u \cdot y-v \cdot x)} e^{2\lambda \eta} \Big( \sum_{k=0}^\infty \|f^\lambda \ast_\lambda \varphi_k^\lambda\|_2^2 \frac{k!(n-1)!}{(k+n-1)!} \varphi_k^\lambda(2iy,2iv) \Big) d\mu(\lambda)
$$
under the assumption that either the left hand side or the right hand side is finite.
\end{thm}

We would like to apply Gutzmer's formula to the holomorphic extension $ U(z,w,\zeta,\rho) $ of the solution $ U = \rho^{2s} f \ast \Phi_{s,\rho} $ of the extension problem. To do this it is  convenient to work with the functions
\begin{equation}
\label{varph}
 \varphi_{s,\delta}(z,a) = \Big((\delta+\frac{1}{4}|z|^2)^2+a^2\Big)^{-\frac{n+1+s}{2}}, \quad z=x+iu.
 \end{equation}
Observe that 
\begin{equation}  
\label{Phisr}
\Phi_{s,\rho}(z,a) = \frac{2^{-(n+1+s)} }{ \pi^{n+1}\Gamma(s)}  \Gamma\Big(\frac{n+1+s}{2}\Big)^2 \varphi_{s,\delta}(z,a)
\end{equation}
 with $ \delta = \frac{1}{4}\rho^2.$ We are therefore led to compute $ f^\lambda \ast_\lambda \varphi_{s,\delta}^\lambda \ast_\lambda \varphi_k^\lambda$. As $ \varphi_{s,\delta}^\lambda $ is radial, we  can write
 \begin{equation*}  
 \varphi_{s,\delta}^\lambda(z) = (2\pi)^{-n} |\lambda|^n \sum_{k=0}^\infty  c_{k,\delta}^\lambda(s) \varphi_k^\lambda(z) 
 \end{equation*}
and  the twisted convolution $\varphi_{s,\delta}^\lambda \ast_\lambda\varphi_k^\lambda$ is a constant multiple of $ \varphi_k^\lambda$, i.e.,
$$
\varphi_{s,\delta}^\lambda \ast_\lambda \varphi_k^\lambda =c_n c_{k,\delta}^\lambda(s) \varphi_k^\lambda(z).
$$
Hence it is enough to calculate the Laguerre coefficients   $ c_{k,\delta}^\lambda(s) $ of $\varphi_{s,\delta}^\lambda(z)$.  These constants are explicitly known and are given  in terms of the auxiliary function defined, for $ a, b \in \R_+ $ and $ c \in \R$, as
\begin{equation*}
L(a,b,c) = \int_0^\infty e^{-a(2x+1)}x^{b-1}(1+x)^{-c} dx.
\end{equation*} 
\begin{prop} [\cite{CRT} Lemma 3.8, \cite{RT} Proposition 3.2]
\label{propck}
For any $ \delta > 0 $ and  $ 0 < s <  \frac{n+1}{2} $ we have
$$  
c_{k,\delta}^\lambda(s) = \frac{(2\pi)^{n+1}|\lambda|^s}{ \Gamma(\frac{n+1+s}{2})^2} L\Big(\delta |\lambda|, \frac{2k+n+1+s}{2}, \frac{2k+n+1-s}{2}\Big).
$$

\end{prop}

 Thus we see that if $ U_\rho(x,u,\xi):= U(x,u,\xi,\rho) =  \rho^{2s} f \ast \Phi_{s,\rho}(x,u,\xi)  $ then 
 $$ 
 U_\rho^\lambda \ast_\lambda \varphi_k^\lambda(x,u)  = \frac{2^{-(n+1+s)} }{ \pi^{n+1}\Gamma(s)}  \Gamma\Big(\frac{n+1+s}{2}\Big)^2  \rho^{2s} c_{k,\rho^2}^{\lambda/4}(s) f^\lambda \ast_\lambda \varphi_k^\lambda(x,u).
 $$
 Hence the task boils down to studying the convergence of the integral (by considering $ (z,w,\zeta) = (iy,iv,\eta)$) 
 \begin{equation}
 \label{identityG} 
 \rho^{4s}\int_{-\infty}^\infty e^{2\lambda \eta} \Big( \sum_{k=0}^\infty \|f^\lambda \ast_\lambda \varphi_k^\lambda\|_2^2 \frac{k!(n-1)!}{(k+n-1)!} (c_{k,\rho^2}^{\lambda/4}(s))^2 \varphi_k^\lambda(2iy,2iv) \Big)\, d\mu(\lambda).
\end{equation}
 We denote  the integrand in \eqref{identityG} by  $ F_\lambda(y,v,\eta) .$  In order to study the convergence of the integral
 $  \int_{|h| < \rho} F_\lambda(h) dh $ where $ |h| $ stands for the homogeneous norm for $ h \in \He^n, $  we  need to investigate the growth of 
 \begin{equation*}    
 \int_{|(y,v,\eta)| < \rho} e^{2\lambda \eta} \varphi_k^\lambda(2iy,2iv) dy\,dv\,d\eta 
 \end{equation*}
and the decay of $ c_{k,\rho^2}^{\lambda/4}(s) $ as functions of $ k$ and $ \lambda.$  These estimates are done in the next subsection.

\subsection{Estimates on Laguerre functions and  the coefficients $ c_{k,\delta}^\lambda(s)$} 

Recall from Subsection~\ref{solvable} that the homogeneous norm of $ (x,u,\xi) \in \He^n $ is given by $ |(x,u,\xi)|^4 
= (|x|^2+|u|^2)^2+16\xi^2.$ The Haar measure on $ \He^n $ has a polar decomposition: there exists a measure $ \sigma_r $ on the Koranyi sphere $ K_r = \{ (y,v,\eta) \in  \He^n : |(y,v,\eta)| = r \} $ such that 
$$  
\int_{\He^n} f(z,a) dz\,da = \int_0^\infty \int_{K_r} f(z,a) d\sigma_r(z,a) \, r^{Q-1} dr.
$$  
 Moreover, $ \sigma_r = \delta_r \sigma_1$, where $\delta_r \varphi(z,\xi)=\varphi(rz,r^2\xi)$. If we let $ \mu_r $ stand for the surface measure on the sphere $ S_r = \{ (z,0)\in\He^n: |z| =r \} $ and $ \delta_t $ for the Dirac measure on $ \R $ supported at the point $ t$, then the measure $ \mu_{r,t} = \mu_r  \ast \delta_t $ is supported on the set $ S_{r,t} = \{ (z,t)\in \He^n: |z| =r \}.$ The measure $ \sigma_r $ can be expressed in terms of the measures $ \mu_{r,t} $ as follows  (see \cite{FH}):
 $$  
 \sigma_r =\frac{\Gamma\big(\frac{n+1}{2}\big)}{\sqrt{\pi}\Gamma\big(\frac{n}{2}\big)} \int_{-\pi /2}^{\pi/2} \mu_{r \sqrt{\cos \theta}, \frac{1}{4}r^2 \sin \theta} \, (\cos \theta)^{n-1} \,d \theta.
 $$
 Integrating the function in \eqref{identityG} using polar coordinates,  we are reduced to proving that 
 $$
 \int_{-\infty}^\infty  \Big( \sum_{k=0}^\infty \|f^\lambda \ast_\lambda \varphi_k^\lambda\|_2^2  (c_{k,\rho^2}^{\lambda/4}(s))^2 \psi_k^\lambda(r)  \Big) d\mu(\lambda) < \infty 
 $$
 where 
 \begin{equation}
 \label{psisp}
 \psi_k^\lambda(r) = \frac{k!(n-1)!}{(k+n-1)!} \int_{ |(y,v,\eta)| = r} e^{2\lambda \eta}  \varphi_k^\lambda(2iy,2iv) d\sigma_r(y,v,\eta).
 \end{equation}
 From the definition of $ \sigma_r $  it follows that 
$$  
\int_{K_r}  f(y,v,-\eta) d\sigma_r = \int_{K_r} f(y,v,\eta)\, d\sigma_r
$$ 
and consequently 
\begin{equation}
\label{simme}
\psi_k^\lambda(r) = \frac{1}{2} \frac{k!(n-1)!}{(k+n-1)!} \int_{ |(y,v,\eta)| = r} \cosh(2\lambda \eta)\,  \varphi_k^\lambda(2iy,2iv) \,d\sigma_r(y,v,\eta).
\end{equation}
In view of the expression for $ \sigma_r $ in terms of $ \mu_{r,t} $ we see that 
\begin{equation}
\label{eq:psi}
\psi_k^\lambda(r) =\frac{\Gamma\big(\frac{n+1}{2}\big)}{\sqrt{\pi}\Gamma\big(\frac{n}{2}\big)}  \frac{k!(n-1)!}{(k+n-1)!} \int_{-\pi /2}^{\pi/2} e^{\frac{1}{2}\lambda r^2 \sin \theta}  \varphi_k^\lambda(2ir \sqrt{\cos \theta}) (\cos \theta)^{n-1} d\theta 
\end{equation}
where we have written 
\begin{equation*}
\varphi_k^\lambda(2ir \sqrt{\cos \theta})
= L_k^{n-1}(-2|\lambda|r^2\cos \theta) e^{|\lambda| r^2 \cos \theta}.
\end{equation*}
By combining \eqref{simme} and \eqref{eq:psi}, we also have 
$$
\psi_k^\lambda(r) = \frac{1}{2}\frac{\Gamma\big(\frac{n+1}{2}\big)}{\sqrt{\pi}\Gamma\big(\frac{n}{2}\big)}  \frac{k!(n-1)!}{(k+n-1)!} \int_{-\pi /2}^{\pi/2} \cosh\Big(\frac{1}{2}\lambda r^2 \sin \theta\Big)  \varphi_k^\lambda(2ir \sqrt{\cos \theta}) (\cos \theta)^{n-1} \,d\theta .
$$
From this expression it follows that $ \psi_k^\lambda(r) $ is an increasing function of $ r.$ Indeed, as $ L_k^\alpha(-s) = \sum_{j=0}^k c_{k,j} s^k $ with non-negative coefficients, the integrand in the above expression for $ \psi_k^\lambda(r) $ is an increasing function of $r$, which proves that $ \psi_k^\lambda(r) $ itself is increasing in $ r$.

In order to estimate $ \psi_k^\lambda(r) $   we need certain asymptotic properties of the Laguerre  functions similar to the estimates obtained in the proof of \cite[Proposition 3.2]{TPac}. Such inequalities are based on the Perron's estimate \cite[Theorem 8.22.3]{Sz}, valid for large $ k$ when $ s< 0$ (the error has a uniform bound for $ s \leq -c$, $c>0$),
\begin{equation}
\label{perron}
L_k^{\alpha}(s)=\frac12\pi^{-1/2}e^{s/2}(-s)^{-\alpha/2-1/4}k^{\alpha/2-1/4}e^{2(-ks)^{1/2}}(1+\mathcal{O}(k^{-1/2})),
\end{equation}
which is valid for $s$ in the complex plane cut along the positive real axis.
The asymptotic properties of Laguerre polynomials in the complex domain lead us to the first estimate.

\begin{lem}
\label{lem:upperpsi} 
As  $(2k+n)|\lambda| \to \infty,$ we have the estimate 
 $$ 
|\psi_k^\lambda(r)| \leq C_n e^{\frac{1}{2} |\lambda| r^2} e^{ 2r \sqrt{(2k+n)|\lambda|}}  .$$
We also have the following lower bound: for any small $ \delta >0 $ and $ |\lambda|r^2 \geq 1,$
$$  
|\psi_k^\lambda(r)|  \geq C_n(\delta) e^{\frac{1}{2} \sin(\pi /4-\delta) |\lambda| r^2} e^{ 2r \cos(\pi /4+\delta)\sqrt{(2k+n)|\lambda|}} ((2k+n)|\lambda|)^{-(2n-1)/4}.
$$
\end{lem}
\begin{proof} The upper bound follows from the asymptotic property of the Laguerre polynomials stated in \eqref{perron}. Indeed, by \eqref{perron}, making use of the fact that $ L_k^\alpha(-s)$, $s> 0 $ is an increasing function of $ s $ we get the estimate
 \begin{equation}
 \label{firste}
\varphi_k^\lambda(2ir \sqrt{\cos \theta})\le C_\delta (2|\lambda| r^2)^{-\frac{n-1}{2}-\frac14}k^{\frac{n-1}{2}-\frac14}e^{2r \sqrt{2k|\lambda|}}
\end{equation}
valid for $ |\lambda| r^2 \geq \delta >0$ and $\theta\in (-\pi/2,\pi/2)$. In order to estimate $\varphi_k^\lambda(2ir \sqrt{\cos \theta})$ for $ |\lambda| r^2 \leq \delta $ we make use of the formula (see \cite[(5.6.5)]{Sz})
$$ 
L_k^\alpha(s) = \frac{(-1)^k \pi^{-1/2}}{\Gamma(\alpha+1/2)} \frac{\Gamma(k+\alpha+1)}{\Gamma(2k+1)} \int_{-1}^1 (1-t^2)^{\alpha-1/2} H_{2k}(\sqrt{s}t) dt, \qquad \alpha>-1/2 
$$
which expresses Laguerre polynomials in terms of Hermite polynomials. The asymptotic properties of Hermite polynomials in the complex plane are given in \cite[(8.22.7)]{Sz}, from which we have
$$ 
\frac{\Gamma(k+1)}{\Gamma(2k+1)} e^{s^2/2}|H_{2k}(is)| \leq C_\delta e^{\sqrt{(4k+1)} s}, \qquad   |s| \leq \delta.
$$
This leads to the estimate 
\begin{equation}
\label{seconde}
 \varphi_k^\lambda(2ir \sqrt{\cos \theta})\le C_\delta \frac{\Gamma(k+n)}{\Gamma(k+1)} e^{2r\sqrt{(2k+1/2)|\lambda|} }\int_{-1}^1 (1-t^2)^{n-3/2} e^{|\lambda|r^2(1-t^2)} dt
 \end{equation}
valid for $ |\lambda| r^2 \leq \delta$ and $\theta\in (-\pi/2,\pi/2)$. Combining \eqref{firste} and \eqref{seconde} we obtain
\begin{equation*}
 \frac{k! (n-1)!}{(k+n-1)!} \varphi_k^\lambda(2ir \sqrt{\cos \theta})\le C e^{2r\sqrt{(2k+n)|\lambda|} }.
 \end{equation*}
Recalling the definition of $ \psi_k^\lambda(r) $ we get the stated upper bound.

On the other hand, for any $\delta>0$ small enough, we get  the following lower bound  using the asymptotic property stated in \eqref{perron}: under the assumption that $ |\lambda| r^2 \geq 1,$
\begin{align*}
\psi_k^{\lambda}(r)&\ge \frac{\Gamma\big(\frac{n+1}{2}\big)}{\sqrt{\pi}\Gamma\big(\frac{n}{2}\big)}  \frac{k!(n-1)!}{(k+n-1)!} \int_{\pi /4-\delta}^{\pi/4+\delta} e^{\frac{1}{2}\lambda r^2 \sin \theta}  \varphi_k^\lambda(2ir \sqrt{\cos \theta}) (\cos \theta)^{n-1} d\theta \\
&\ge C_n \frac{k!(n-1)!}{(k+n-1)!} e^{\frac{1}{2} \sin(\pi /4-\delta) |\lambda| r^2} \int_{\pi /4-\delta}^{\pi/4+\delta} \varphi_k^\lambda(2ir \sqrt{\cos \theta}) (\cos \theta)^{n-1} d\theta\\
&\ge C_n(\delta) e^{\frac{1}{2} \sin(\pi /4-\delta) |\lambda| r^2} e^{ 2r (\cos(\pi /4+\delta)\sqrt{(2k+n)|\lambda|}} \big((2k+n)|\lambda|\big)^{-(2n-1)/4}.
\end{align*}

The proof is complete.

\end{proof}

In the next lemma we obtain good estimates on the coefficients $ c_{k,\delta}^\lambda(s) $. 
\begin{lem}
\label{lem:estimatec}
For $0<s\le 1/2$, $\rho>0$ and $\lambda\in \R$, we have the following two estimates, as $(2k+n)|\lambda| \to \infty$, 
$$ 
  c_{k,\rho^2}^{\lambda /4}(s)
 \leq  C_{n,s} \rho^{-2s}e^{-\frac12 \rho \sqrt{(2k+n+1-s)|\lambda|}}, \qquad c_{k,\rho^2}^{\lambda /4}(s) \leq C_{n,s} \rho^{-2s} e^{-\frac14 |\lambda| \rho^2}.
 $$
 Combining the above two estimates, we have, for  any $ 0 < \gamma < 1/2,$ 
 $$ 
  (c_{k,\rho^2}^{\lambda /4}(s))^2
 \leq  C_{n,s} \rho^{-4s}e^{-2\gamma \rho \sqrt{(2k+n+1-s)|\lambda|}}  e^{-\frac12 (1-2\gamma) |\lambda| \rho^2}.
 $$
\end{lem}
\begin{proof}
We write $ L(a,b,c) $ in terms of the confluent hypergeometric function of second type $ U(a,b,c) $ as follows (see \cite[p. 19]{CRT})
$$ 
L(a,b,c) =e^{-a} \Gamma(b) U(b,b-c+1,2a).
$$
Let us call $ \alpha := (n+1+s)/2$, $ \beta := (n+1-s)/2$.
Take $ a = \delta|\lambda|$, $b = k+\alpha$, $c = k+\beta $, then we have
\begin{equation}
\label{LU}
L(\delta|\lambda|,k+\alpha,k+\beta) =  e^{-\delta |\lambda|} \Gamma(k+\alpha) U(k+\alpha, \alpha-\beta+1,2\delta|\lambda|).
\end{equation}
On the other hand, the following asymptotic properties of $ U(a,b,z) $ are known (see \cite[13.8.8]{OlMax}, also \cite{T}): when $a\to \infty$ and $b\le 1$ fixed,
\begin{equation}
\label{uas}
U(a,b,x) = \frac{2e^{x/2}}{\Gamma(a)}\bigg(\sqrt{\frac{2}{\gamma}\tanh\Big(\frac{w}{2}\Big)}\Big(\frac{1-e^{-w}}{\gamma}\Big)^{-b}\gamma^{1-b}K_{1-b}(2\gamma a)+a^{-1}\Big(\frac{a^{-1}+\gamma}{1+\gamma}\Big)^{1-b}e^{-2\gamma a}\mathcal{O}(1)\bigg),
\end{equation} 
where $w=\arccosh(1+(2a)^{-1}x)$, and  $\gamma=(w+\sinh w)/2$. Since in \eqref{uas} we have $b= \alpha-\beta+1 = s+1>1$, we use the transformation (see \cite[13.2.40]{OlMax})
$$
U(a,b,z)=z^{1-b}U(a-b+1,2-b,z).
$$ 
Then,
\begin{equation}
\label{us}
 U(k+\alpha, \alpha-\beta+1,2\delta|\lambda|)=(2\delta|\lambda|)^{-s}U(k+\beta,1-s,2\delta|\lambda|).
\end{equation}
So, as $k\to\infty$, in view of \eqref{LU}, \eqref{us} and \eqref{uas}
\begin{multline*}
L(\delta|\lambda|,k+\alpha,k+\beta)=e^{-\delta|\lambda|}\Gamma(k+\alpha)(2\delta|\lambda|)^{-s} \frac{2e^{\delta|\lambda|}}{\Gamma(k+\beta)}\\\times\bigg(\sqrt{\frac{2}{\gamma}\tanh\Big(\frac{w}{2}\Big)}\Big(\frac{1-e^{-w}}{\gamma}\Big)^{s-1}\gamma^{s}K_{s}(2\gamma (k+\beta))
+\frac{1}{k+\beta}\Big(\frac{(k+\beta)^{-1}+\gamma}{1+\gamma}\Big)^{s}e^{-2\gamma (k+\beta)}\mathcal{O}(1)\bigg),
\end{multline*}
where we chose, in \eqref{uas}, $a:=k+\beta$, $b:=1-s$ and $x:=\delta|\lambda|$, so that we have 
$
\cosh w=1+\frac{x}{2a}=1+\frac{\delta|\lambda|}{k+\beta}$, 
and 
$$
2\gamma=w+\sinh w=\arccosh\Big(1+\frac{\delta|\lambda|}{k+\beta}\Big)+\sinh\Big(\arccosh\Big(1+\frac{\delta|\lambda|}{k+\beta}\Big)\Big).
$$
We will first show that $ \gamma a \geq c \sqrt{ax} $ for some $ c> 0$  as $ a $ tends to infinity, or equivalently that $ \gamma(k+\beta) \geq c \sqrt{(k+\beta)\delta|\lambda|} $ as $k$ tends to infinity. It is enough to prove $ \frac{a \gamma^2}{x}  \geq c^2$. Consider 
\begin{equation}
\label{comput}
\frac{a \gamma^2}{x} = \frac{1}{x} \frac{x}{2(\cosh w-1)} \frac{(w+\sinh w)^2}{4} = \frac18 \frac{(w+\sinh w)^2}{(\cosh w-1)}.
\end{equation}
Then
$$
\frac18 \frac{(w+\sinh w)^2}{(\cosh w-1)}\ge \frac18 \frac{\sinh^2w}{\cosh w-1}= \frac18 \frac{\sinh w}{\tanh \frac{w}{2}}=\frac14 \cosh^2\frac{w}{2}\ge \frac14.
$$
Thus our claim is proved with $ c = \frac12$, i.e., we have proved that
\begin{equation}
\label{eq:estimateag}
\gamma a \geq \frac12 \sqrt{ax}.
\end{equation}
Observe that we could do even a bit better. Indeed, observe that, when $k\to \infty$ and $\delta$ and $|\lambda|$ are fixed, we have that $w\to 0$, so that $\gamma \sim w \sim \sinh{w}$. Then, in \eqref{comput} we can write
$$
\frac18 \frac{(w+\sinh w)^2}{(\cosh w-1)}\sim \frac12 \frac{(\sinh w)^2}{(\cosh w-1)}= \frac12 \frac{(\sinh w)^2}{\tanh \frac{w}{2}}=\cosh^2\frac{w}{2}\sim 1.
$$

Now we will estimate the terms in front of $ K_{1-b}(2\gamma a) $ and $ e^{-2\gamma a} $ (i.e. $ K_{s} $ and $ e^{-2\gamma (k+\beta)}$). On one hand, in view of \eqref{eq:estimateag},
$$
\frac{a^{-1}+\gamma}{1+\gamma}=\frac{(\gamma a)^{-1}+1}{\gamma^{-1}+1}\le (\gamma a)^{-1}+1\le 2(\sqrt{ax})^{-1}+1.
$$
Therefore, $a^{-1}\big(\frac{a^{-1}+\gamma}{1+\gamma}\big)^{s}\le a^{-1}(2(\sqrt{ax})^{-1}+1)^s$.
So, as $a x\to +\infty$, 
\begin{equation}
\label{primer}
a^{-1}\Big(\frac{a^{-1}+\gamma}{1+\gamma}\Big)^{s}e^{-2\gamma a}\le a^{-1}e^{-2\gamma a}\le a^{-1}e^{-\sqrt{ax} }.
\end{equation}
Again we could also get a lower bound with a pay of $a^{-s}$. Indeed, 
$$
\frac{a^{-1}+\gamma}{1+\gamma}\ge a^{-1}\frac{1}{1+\gamma}\ge c a^{-1},
$$
since $\gamma\to 0$ as $a\to \infty$. Thus, altogether (observe the difference with the upper bound \eqref{primer}),
$$
a^{-1}\Big(\frac{a^{-1}+\gamma}{1+\gamma}\Big)^{s}e^{-2\gamma a}\ge a^{-1}a^{-s}e^{-2\gamma a}\sim a^{-1-s}e^{-2\sqrt{ax} }.
$$
On the other hand, it remains to take care of the term in front of $ K_{1-b}(2 \gamma a)$:
$$
\sqrt{\frac{2}{\gamma}\tanh\Big(\frac{w}{2}\Big)}\Big(\frac{1-e^{-w}}{\gamma}\Big)^{s-1}\gamma^{s}.
$$ 
By the Taylor expansions of $\sinh w$ and $e^{-w}$, we have that $w\sim \sinh w\sim 1- e^{-w}$. Hence,  $\gamma=w+\sinh w\sim w$ as $a\to +\infty$. Moreover,
$$
\tanh\Big(\frac{w}{2}\Big)=\frac{1-e^{-w}}{1+e^{-w}}\le 1-e^{-w}\sim w. 
$$
From here, we conclude that 
\begin{equation}
\label{segun}
\sqrt{\frac{2}{\gamma}\tanh\Big(\frac{w}{2}\Big)}\Big(\frac{1-e^{-w}}{\gamma}\Big)^{s-1}\sim C, \quad \text{as } a\to +\infty.
\end{equation}
Finally, since  $K_{1-b}(2\gamma a)\sim e^{-2\gamma a}(2\gamma a)^{-1/2}$ as $a\to +\infty$, we have 
\begin{equation}
\label{tercer}
\gamma^se^{-2\gamma a}(2\gamma a)^{-1/2}= 2^{-s}\frac{(2\gamma a)^s}{a^s}e^{-2\gamma a}(2\gamma a)^{-1/2}\le (\sqrt{ax})^{s-1/2}a^{-s}e^{-2\gamma a}\le a^{-s}e^{-\sqrt{ax} }
\end{equation}
where we have used the fact that $ 2 \gamma a \ge \sqrt{ax}$, $ax \to \infty$ as $ (2k+n)|\lambda| \to \infty $ and $s\le 1/2$. Observe that here we are getting an upper bound. 
So, in view of \eqref{primer}, \eqref{segun} and \eqref{tercer}, we get
$$
L(\delta|\lambda|,k+\alpha,k+\beta)\le C\frac{\Gamma(k+\alpha)}{\Gamma(k+\beta)}(2\delta|\lambda|)^{-s}\Big[\frac{e^{- \sqrt{(2k+n+1-s)\delta|\lambda|} }}{(k+\beta)^s}+\frac{e^{- \sqrt{(2k+n+1-s)\delta|\lambda|} }}{k+\beta}\Big],
$$
thus, as $k\to +\infty$, taking into account the asymptotics for the quotient of Gamma functions $\frac{\Gamma(k+x)}{\Gamma(k+y)}\sim k^{x-y}$, we conclude
$$
L(\delta|\lambda|,k+\alpha,k+\beta)\le C(k+\beta)^s(2\delta|\lambda|)^{-s}\frac{1}{(k+\beta)^s}e^{- \sqrt{(2k+n+1-s)\delta|\lambda|} }=C (\delta|\lambda|)^{-s}e^{- \sqrt{(2k+n+1-s)\delta|\lambda|} }.
$$
From here and in view of Proposition \ref{propck} we obtain
\begin{equation}
\label{eq:estimateck1}
c_{k,\rho^2}^{\lambda/4}(s)\le C_{n,s}\rho^{-2s} e^{-\frac12\rho \sqrt{(2k+n+1-s)|\lambda|} }.
\end{equation}
We also have the estimate
\begin{align*} 
L(\delta|\lambda|, k+\alpha,k+\beta) &= e^{-\delta |\lambda|} \int_0^\infty  e^{-2\delta |\lambda|t} t^{k+\alpha-1}(1+t)^{-k-\beta} dt\\
& \leq e^{-\delta |\lambda|} \int_0^\infty  e^{-2\delta |\lambda|t} t^{s-1} dt
 = \Gamma(s) (2\delta|\lambda|)^{-s} e^{-\delta |\lambda|}.
\end{align*}
This gives another estimate
\begin{equation}
\label{eq:estimateck2}
c_{k,\rho^2}^{\lambda /4}(s) \leq C  \rho^{-2s} e^{-\frac{1}{4}|\lambda|\rho^2}.
\end{equation}
Combining the two estimates \eqref{eq:estimateck1} and \eqref{eq:estimateck2} for $ c_{k,\rho^2}^{\lambda/4} $ we get
$$  
(c_{k,\rho^2}^{\lambda /4}(s))^2 \leq C   \rho^{-4s}  e^{-\frac12\rho \sqrt{(2k+n+1-s)|\lambda|}} e^{-\frac{1}{4}|\lambda|\rho^2} .$$
This completes the proof of the lemma.
\end{proof} 

\subsection{Certain $ L^2 $ spaces of holomorphic functions}
\label{certainL2}

We now consider a domain in $ \C^n \times \C^n \times \C $ with the key property that it is invariant under the action of the Heisenberg motion group $ G_n.$  For each $ r > 0 $, recall the definition in \eqref{tube}
$$ 
\Omega_r = \Big\{  (z,w,\zeta) = (x+iy,u+iv,\xi+i\eta):  (|y|^2+|v|^2)^2+16 \big( \eta-\frac{1}{2}(u\cdot y-v\cdot x)\big)^2 < r^4 \Big\}.
$$ 
Let us first check that $ \Omega_r $ is indeed invariant under the action of $ G_n$. When $ g = (x',u',\xi') \in \He^n $ we have 
$$ 
g\cdot(z,w,\zeta) =(x',u',\xi')(z,w,\zeta) = \big(x'+x+iy, u'+u+iv, \xi'+\xi+i\eta+ \frac{1}{2}(u'\cdot z-x'\cdot w)\big).
$$ 
Therefore,
$$ 
(|y|^2+|v|^2)^2 + 16 \big( \eta+ \frac{1}{2}(u'\cdot y-x'\cdot v) -\frac{1}{2}(u+u')\cdot y- (x+x')\cdot v\big)^2 = \big (|y|^2+|v|^2)^2+16 ( \eta-\frac{1}{2}(u\cdot y-v\cdot x)\big)^2
$$
which proves that $ g\cdot(z,w,\zeta) \in \Omega_r $ whenever $ (z,w,\zeta) \in \Omega_r.$ We need to check the same for $ g =(0,\sigma)\in G_n$, $\sigma \in U(n)$. The action of $ \sigma = a+ib $  on  $ \C^n \times \C^n \times \C $ is given by $ \sigma\cdot(z,w,\zeta) = (az-bw,aw+bz, \zeta)$. In real coordinates, we have
$$ 
\sigma\cdot (x+iy,u+iv,\xi+i\eta) = \big((ax-bu)+i(ay-bv),(au+bx)+i(av+by),\xi+i\eta\big).
$$ 
Thus we need to check that
$$ ( |(ay-bv)|^2+|(av+by)|^2)^2+16 \Big( \eta -\frac{1}{2} \big((au+bx)\cdot (ay-bv)- (ax-bu)\cdot (av+by)\big)\Big)^2$$
is the same as $  (|y|^2+|v|^2)^2+16 \big( \eta-\frac{1}{2}(u\cdot y-v\cdot x)\big)^2.$ As 
$$ 
(ay-bv)+i(av+by) = (a+ib)(y+iv) =\sigma\cdot (y+iv) 
$$ 
with $ \sigma \in U(n) $ it is clear that 
$|(ay-bv)|^2+|(av+by)|^2 = |y|^2+|v|^2.$ If we let $ [(x,u), (y,v)] = (u\cdot y- v\cdot x) $ stand for the symplectic form on $ \R^{2n} $ (see Subsection \ref{solvable}), it can be checked that $ [\sigma\cdot (x,u), \sigma\cdot (y,v)] = [(x,u), (y,v)]$ and consequently
$$ 
(au+bx)\cdot (ay-bv)- (ax-bu)\cdot(av+by) = [\sigma\cdot(x,u), \sigma \cdot(y,v)] = [(x,u), (y,v)] .
$$

\begin{defn}
 Consider $ \mathcal{O}(\C^{2n+1}),$  the space of all holomorphic functions on $\C^{2n+1}$
and equip it with the $ L^2 $ norm
\begin{equation}
\label{norm}
\| F\|^2 = \int_{\Omega_{\rho}}\big|F(z,w,\zeta)\big|^2\,dz\,dw\,d\zeta.
\end{equation}
We will denote by $\widetilde{\mathcal{H}}(\Omega_{\rho})$ the completion of $ \mathcal{O}(\C^{2n+1})$ with respect to the norm \eqref{norm}. 
\end{defn}
 It is easy to see that
$
\widetilde{\mathcal{H}}(\Omega_{\rho})\subset \mathcal{O}(\Omega_{\rho})\cap L^2(\Omega_{\rho}).
$
 Indeed, if  $ F \in \widetilde{\mathcal{H}}(\Omega_{\rho}),$ then there exists a sequence $\{F_j\} $
 with $F_j\in \mathcal{O}(\C^{2n+1})$ such that 
 $$
 \int_{\Omega_{\rho}}|F_j(z,w,\zeta)-F(z,w,\zeta)|^2\,dz\,dw\,d\zeta \to 0\quad \text{ as } j\to \infty. 
 $$
 By the mean value property of holomorphic functions
 $$
 F_j(z,w,\zeta)=\frac{1}{|D_{\delta}(z,w,\zeta)|}\int_{D_{\delta}(z,w,\zeta)} F_j(z',w',\zeta')\,dz'\,dw'\, d\zeta',
 $$
 where $D_{\delta}(z,w,\zeta):=B_{\delta}(z)\times B_{\delta}(w)\times B_{\delta}(\zeta)$ is a polydisc contained in $ \Omega_\rho$, and we can check that, for $K\subset \Omega_{\rho}$ compact,  
 \begin{align*}
 \sup_{K\subset \Omega_\rho} | F_j(z,w,\zeta)| &= \sup_{ (z,w,\zeta) \in K } \frac{1}{|D_{\delta}(z,w,\zeta)|}\Big|\int_{D_{\delta}(z,w,\zeta)} F_j(z',w',\zeta')\,dz'\,dw'\,d\zeta'\Big|\\
 &\le C_\delta \Big(\int_{\Omega_{\rho}}| F_j(z',w',\zeta')|^2\,dz'\,dw'\,\,d\zeta' \Big)^{1/2}.
 \end{align*} 
Hence $F_j$ is uniformly Cauchy and converges to $F $ over compact subsets. This shows that $F $ is holomorphic.
Let us introduce one more space $ \mathcal{H}_\rho(\He^n)$,  a subspace of $ L^2(\He^n).$    We let
\begin{equation}
\label{Psiint} 
\Psi_k^\lambda(\rho) = c_n \int_0^{\rho} \psi_k^\lambda(r) r^{Q-1} \,dr,
\end{equation}
where  $\psi_k^{\lambda}(r)$ is  the function already defined in \eqref{psisp}. In other words,
\begin{equation}
\label{Psi}  
\Psi_k^\lambda(\rho) =  \frac{k!(n-1)!}{(k+n-1)!} \int_{ |(y,v,\eta)| \leq \rho} e^{2\lambda \eta}  \varphi_k^\lambda(2iy,2iv) dy\,dv\,d\eta.
\end{equation} 

\begin{defn}
\label{defHrho}
We say that $ f \in   \mathcal{H}_\rho(\He^n)$ if 
$$
\int_{-\infty}^\infty  \Big( \sum_{k=0}^\infty \|f^\lambda \ast_\lambda \varphi_k^\lambda\|_2^2 \, \Psi_k^\lambda(\rho) \Big)\, d\mu(\lambda) < \infty
$$
where $ \Psi_k^\lambda$ is the one in \eqref{Psi}.
\end{defn}

\begin{thm}
\label{thm:general}
Let $ \rho > 0.$ A function $ F $  from $  L^2(\Omega_\rho) $ belongs to  $\widetilde{\mathcal{H}}(\Omega_{\rho})$ if and only if its restriction $ f $ belongs to $  \mathcal{H}_\rho(\He^n). $ Moreover, 
$$  
\int_{\Omega_{\rho}} | F(z,w,\zeta)|^2 \,dz\,dw\,d\zeta  = C_{n}  \int_{-\infty}^\infty  \Big( \sum_{k=0}^\infty \|f^\lambda \ast_\lambda \varphi_k^\lambda\|_2^2 \,\, \Psi_k^\lambda(\rho) \Big) d\mu(\lambda).
$$
\end{thm}
\begin{proof}  For $ F \in \widetilde{\mathcal{H}}(\Omega_{\rho})$ consider its norm squared
$$  
\int_{\Omega_\rho} |F(z,w,\zeta)|^2 \,dz\,dw\,d\zeta = \int_{\He^n} \Big(\int_{\Omega_\rho(x,u,\xi)} |F(x+iy,u+iv,\xi+i\eta)|^2 \,dy\,dv\,d\eta\Big) \,dx\,du\,d\xi
$$ 
where 
$$ 
\Omega_\rho(x,u,\xi) =\big \{ (y,v,\eta): (|y|^2+|v|^2)^2+16 \big( \eta-\frac{1}{2}[(x,u),(y,v)]\big)^2 < \rho^4\big \}.
$$
By making a change of variables in $ \eta $ the above integral takes the form
\begin{multline*}
\int_{\He^n} \Big(\int_{\Omega_\rho(0)} \big|F\big(x+iy,u+iv,\xi+i(\eta+\frac{1}{2}[(x,u),(y,v)])\big)\big|^2\, dy\,dv\,d\eta\Big)\, dx\,du\,d\xi\\
= \int_{\He^n} \Big(\int_{\Omega_\rho(0)} \big|F\big((x,u,\xi)(iy,iv,i\eta)\big)\big|^2\, dy\,dv\,d\eta\Big) \,dx\,du\,d\xi,
\end{multline*}
where we used that 
\begin{equation}
\label{produ}
(x,u,\xi)\cdot (iy,iv,i\eta) = \big(z,w, \zeta+\frac{1}{4}(z\cdot \bar{w}-w\cdot\bar{z})\big)= \big(z,w, \zeta+\frac{i}{2}[(x,u),(y,v)]\big),
\end{equation}
which folllows from the definition \eqref{law}.
As the Lebesgue measure $ dy \,dv $ is invariant under the change of variables $ (y,v) \rightarrow \sigma\cdot(y,v)$, $\sigma \in U(n) $, the last integral can be written as
$$
\int_{\He^n} \Big(  \int_{U(n)}\int_{\Omega_\rho(0)} |F((x,u,\xi)(i \sigma\cdot (y,v),i\eta))|^2 \,dy\,dv\,d\eta\,  d\sigma \Big) \,dx\,du\,d\xi
$$
which proves that 
$$  \int_{\Omega_{\rho}} | F(z,w,\zeta)|^2 \,dz\,dw\,d\zeta  = \int_{\Omega_\rho(0)} \Big(  \int_{G_n} | F(g\cdot(iy,iv,i\eta)|^2 dg\Big) \,dy\,dv\,d\eta.
$$
The inner integral can be evaluated using Gutzmer's formula in Theorem \ref{thm:gutzmer}, which gives
$$ 
\int_{G_n} | F(g\cdot(iy,iv,i\eta)|^2 dg=c_n \int_{-\infty}^\infty e^{2\lambda \eta} \Big( \sum_{k=0}^\infty \|F^\lambda \ast_\lambda \varphi_k^\lambda\|_2^2 \frac{k!(n-1)!}{(k+n-1)!} \varphi_k^\lambda(2iy,2iv) \Big) d\mu(\lambda).
$$
Since $ \Omega_\rho(0) = \{ h \in \He^n: |h| < \rho \} $, integrating in polar coordinates and recalling the definition of $ \Psi_k^\lambda(\rho) $ in \eqref{Psi} we obtain 
$$ 
 \int_{\Omega_{\rho}} | F(z,w,\zeta)|^2 \,dz\,dw\,d\zeta  = C_{n}  \int_{-\infty}^\infty  \Big( \sum_{k=0}^\infty \|f^\lambda \ast_\lambda \varphi_k^\lambda\|_2^2 \,\, \Psi_k^\lambda(\rho) \Big) d\mu(\lambda).
$$

In order to prove the converse, we consider the full Laplacian on the Heisenberg group $\Delta:=-\partial_\xi^2+\mathcal{L}$. Note that this is a non-negative operator. The holomorphic extension of solutions of the heat equation associated to this operator has been studied in \cite{KTX}. It is known that for $ f \in L^2(\He^n) $ the function 
$e^{-t\Delta}f(x,u,\xi) $ has an entire extension to $\C^{2n+1}$ which we call $ F_t(z,w,\zeta)$.  Then Gutzmer's formula from Theorem \ref{thm:gutzmer} is valid and we have
\begin{multline}
\label{eq:gutheat}
\int_{G_n} |F_t(g\cdot(iy,iv,i\eta))|^2 \,dg  \\
 = c_n \int_{-\infty}^\infty e^{2\lambda \eta} \Big( \sum_{k=0}^\infty \|f^\lambda \ast_\lambda \varphi_k^\lambda\|_2^2 e^{-2t\lambda^2}e^{-2t(2k+n)|\lambda|}\frac{k!(n-1)!}{(k+n-1)!} \varphi_k^\lambda(2iy,2iv) \Big)\, d\mu(\lambda).
\end{multline}
The right hand side of \eqref{eq:gutheat} reads, after integration over  $\Omega_{\rho}(0)$, in view of \eqref{Psi}:
$$
C_{n}  \int_{-\infty}^\infty  \Big( \sum_{k=0}^\infty \|f^\lambda \ast_\lambda \varphi_k^\lambda\|_2^2 \, e^{-2t(\lambda^2+(2k+n)|\lambda|)}\Psi_k^\lambda(\rho) \Big)\, d\mu(\lambda).
$$
The left hand side of \eqref{eq:gutheat} reads, after integration over $\Omega_{\rho}(0)$:
\begin{align*}
&\int_{|(y,v,\eta)|<\rho}\int_{U(n)}\int_{\He^n}|F_t((x,u,\xi,k)\cdot (iy,iv,i\eta))|^2\,dx\,du\,d\xi\,dkdy\,dv\,d\eta\\
&\qquad=\int_{|(y,v,\eta)|<\rho} \int_{U(n)} \int_{\He^n}|F_t((x,u,\xi)\cdot (k(iy,iv),i\eta)|^2\,dx\,du\,d\xi\, dk\, dy\,dv\,d\eta\\
&\qquad = \int_{|(y,v,\eta)|<\rho}\int_{\He^n}\Big|F_t(z,w,\zeta+\frac{i}{2}[(x,u),(y,v)]))\Big|^2\,dz\,dw\,d\zeta
\end{align*}
where in the second equality we have used again the fact that the measure $dy\,dv$ is  invariant under the action of the unitary group $ U(n) $ and \eqref{produ}.  Thus we have proved the identity
\begin{equation}  
\label{omegar}
\int_{\Omega_{\rho}} | F_t(z,w,\zeta)|^2 \,dz\,dw\,d\zeta  = C_{n,s}  \int_{-\infty}^\infty  \Big( \sum_{k=0}^\infty \|f^\lambda \ast_\lambda \varphi_k^\lambda\|_2^2 \,\, e^{-2t(\lambda^2+(2k+n)|\lambda|)}\Psi_k^\lambda(\rho) \Big) d\mu(\lambda).
\end{equation}

 Now let us take a sequence $t_j\to 0$  and call $F_j:=F_{t_j}$. In view of the identity \eqref{omegar} we write
  \begin{multline*}
 \int_{\Omega_{\rho}} \Big|F_j(z,w,\zeta)-F_m(z,w,\zeta)\Big|^2 \,dz\,dw\,d\zeta
 \\= c_n \int_{-\infty}^\infty  \Big( \sum_{k=0}^\infty \|f^\lambda \ast_\lambda \varphi_k^\lambda\|_2^2 \big|e^{-2t_j\lambda^2}e^{-2t_j(2k+n)|\lambda|}-e^{-2t_m\lambda^2}e^{-2t_m(2k+n)|\lambda|}\big|  \Psi_k^\lambda(\rho)\Big) d\mu(\lambda).
 \end{multline*}
  If we assume that the function $ f \in \mathcal{H}_\rho(\He^n) $, it follows that $ F_j $ is Cauchy in $ L^2(\Omega_\rho) $ and hence converges to an $ F \in \widetilde{\mathcal{H}}(\Omega_\rho)$ which gives the holomorphic extension of $ f $, since $ F_t \rightarrow f $ in $ L^2(\He^n).$ Moreover, by taking limit as $ t $ goes to zero in \eqref{omegar} we obtain the equality of norms.
 \end{proof} 
 
\subsection{Holomorphic extensions of solutions of the extension problem}

Theorem \ref{thm:general} has two interesting consequences.  First of all, when  $ F = e^{-t\sqrt{\Delta}}f$, for $f \in L^2(\He^n) $,    it follows that $ e^{-t\sqrt{\Delta}}f \in  \mathcal{H}_{t/\sqrt{2}}(\He^n).$ This is a consequence of the estimate  $ (|\lambda|^2+ (2k+n)|\lambda|) \geq \frac{1}{2} (|\lambda|+\sqrt{(2k+n)|\lambda|})^2$ and the growth estimates on $ \psi_k^\lambda(r).$ We therefore have the following result  on Poisson integrals associated to $ \Delta $ on $ \He^n.$ Compare this with the result obtained in \cite{TPac}.

\begin{thm}  Fix $ t >0.$ Then for any $ f \in L^2(\He^n) $ the function  $ F = e^{-t\sqrt{\Delta}}f $ extends to $ \Omega_{t/\sqrt{2}}$ as a holomorphic function $\widetilde{F}$ and satisfies the estimate
$$  \int_{\Omega_{t/\sqrt{2}}} |\widetilde{F}(z,w,\zeta)|^2 \,dz\,dw\,d\zeta \leq C   \int_{\He^n} |f(x,u,\xi)|^2 dx du d\xi.$$
\end{thm}

The second consequence concerns the heat semigroup $ e^{-t\Delta} $, in which case we can do better. When $ f \in L^2(\He^n) $ it is clear that  $ F = e^{-t\Delta}f $ belongs to $ \mathcal{H}_\rho(\He^n)$ for any $ \rho > 0.$ Moreover, $ F $ extends to $ \C^{2n+1} $ as an entire function, see \cite{KTX}. 

\begin{thm}  
\label{entire}
Fix $ t >0.$ Then for any $ f \in L^2(\He^n) $ the function $ F = e^{-t\Delta}f $ has an entire extension $\widetilde{F}$ to $ \C^{2n+1} $ and for each $ r>0$ there exists a constant $ C(r) $ such that
$$  
\int_{\Omega_{r}} |\widetilde{F}(z,w,\zeta)|^2 \,dz\,dw\,d\zeta \leq C(r)   \int_{\He^n} |f(x,u,\xi)|^2 \,dx \,du \,d\xi .
$$ 
Moreover, we also have
$$ 
\int_{\C^{2n+1}} |\widetilde{F}(z,w,\zeta)|^2 e^{-\frac{2}{t} |\Im(z,w,\zeta)|^4} \,dz\,dw\,d\zeta \leq C   \int_{\He^n} |f(x,u,\xi)|^2 \,dx\, du \,d\xi.
$$
\end{thm}

\begin{proof} The proof, in Theorem \ref{thm:general}, that $ F $ has a holomorphic  extension $ \widetilde{F} \in \widetilde{\mathcal{H}}(\Omega_r)$ actually gives the identity
\begin{multline*}
 \int_{\He^n} \Big( \int_0^{\tau}  \int_{|b| =r} |\widetilde{F}(a+ib)|^2  r^{Q-1} \,d\sigma_r(b) \,dr \Big) \,da\\
  = C_{n}  \int_{-\infty}^\infty  \Big( \sum_{k=0}^\infty \|F^\lambda \ast_\lambda \varphi_k^\lambda\|_2^2 \,\int_0^{\tau} \psi_k^\lambda(r) r^{Q-1}\, dr  \Big)\, d\mu(\lambda).
\end{multline*}
Differentiating the above with respect to $ \tau$ we obtain the following identity,  valid for any $ r > 0 $:
\begin{equation}
\label{ident}
 \int_{\He^n} \Big(  \int_{|b| =r} |\widetilde{F}(a+ib)|^2   d\sigma_r(b)  \Big) \,da  = C_{n}  \int_{-\infty}^\infty  \Big( \sum_{k=0}^\infty \|f^\lambda \ast_\lambda \varphi_k^\lambda\|_2^2  w_k^\lambda(r) \Big)\, d\mu(\lambda)
\end{equation}
where $ w_k^\lambda(r) = e^{-2t |\lambda|^2} e^{-2t((2k+n)|\lambda|)} \psi_k^\lambda(r).$ In view of the estimate for $ \psi_k^\lambda(r) $ proved in Lemma \ref{lem:upperpsi},  by maximising the factors $ 
e^{-2t |\lambda|^2} e^{\frac{1}{2}|\lambda|r^2}$ and $ 
e^{-2t((2k+n)|\lambda|)}e^{2r \sqrt{(2k+n)|\lambda|}} $ separately, we can easily verify that
\begin{equation}
\label{maxim}
w_k^\lambda(r) \leq C e^{\frac{r^2}{2t}} e^{\frac{r^4}{32 t}} \leq C e^{\frac{r^4}{t}}.
\end{equation}
Integrating the identity \eqref{ident}   with respect to the measure $ r^{Q-1} e^{-\frac{2}{t} r^4} dr $ and using the estimate on $ w_k^\lambda(r) $ we complete the proof.
\end{proof}

Using the estimates proved in Lemma \ref{lem:estimatec} on the function $c_{k,\rho^2}^{\lambda/4}(s) $ we can now prove the following result on solutions of the extension problem \eqref{epg}. Recall the function $\Phi_{s,\rho}(\cdot)$ in  \eqref{Phi}.

\begin{thm}
\label{thm:solut}
For $ 0 < s \le 1/2 $, let  $ U(\cdot,\rho) =  \rho^{2s} f \ast \Phi_{s,\rho}(\cdot)$ where $ f \in L^2(\He^n) .$ Then for any $ 0 < \gamma \leq -1+\sqrt{2} $ the solution of the extension problem \eqref{epg} extends to $ \Omega_{ \gamma \rho} $ as a holomorphic function $ \widetilde{U}$, belongs to $ \widetilde{\mathcal{H}}(\Omega_{\gamma \rho}) $ and satisfies the uniform estimate  
$\| \widetilde{U}(\cdot,\rho) \|_{\widetilde{\mathcal{H}}(\Omega_{ \gamma \rho }) } \leq C \rho^{Q/2} \|  f\|_2$ for all $ \rho >0 $. 
\end{thm}
\begin{proof}
 Let us suppose that the solution of the extension problem is given by  $U(x,u,\xi,\rho)=\rho^{2s}f\ast \Phi_{s,\rho}(x,u,\xi)$ for some $f\in L^2(\He^n)$.  With the notation  $ U_\rho(x,u,\xi) := U(x,u,\xi, \rho) $ we have the relation 
  $$ 
  U_\rho^\lambda \ast_\lambda \varphi_k^\lambda(x,u)  = \frac{2^{-(n+1+s)} }{ \pi^{n+1}\Gamma(s)}  \Gamma\Big(\frac{n+1+s}{2}\Big)^2  \rho^{2s} c_{k,\rho^2}^{\lambda/4}(s) f^\lambda \ast_\lambda \varphi_k^\lambda(x,u).
  $$ 
 Taking \eqref{Psiint}  into account and in view of the estimates proved in Lemmas \ref{lem:upperpsi} and \ref{lem:estimatec} it follows that
\begin{equation} 
\label{gesti}
 \rho^{4s} (c_{k,\rho^2}^{\lambda/4}(s))^2 \Psi_k^\lambda(\gamma \rho) \leq C \rho^Q  e^{\frac{1}{2} \gamma^2 |\lambda|\rho^2} e^{-\frac12 (1-2\gamma) |\lambda| \rho^2}
\end{equation}
  which is bounded by  $  C \rho^Q $ since $ \gamma^2 \leq 1-2\gamma $ under the assumption on $ \gamma.$
  This proves that
  $$ \int_{-\infty}^\infty  \Big( \sum_{k=0}^\infty \|U_\rho^\lambda \ast_\lambda \varphi_k^\lambda\|_2^2 \,\, \Psi_k^\lambda(\gamma \rho) \Big) d\mu(\lambda)\leq C \rho^Q  \int_{-\infty}^\infty  \Big( \sum_{k=0}^\infty    \|f^\lambda \ast_\lambda \varphi_k^\lambda\|_2^2  \Big) d\mu(\lambda) = C \rho^Q \|f\|_2^2.$$
  By Theorem \ref{thm:general} the solution $ U(x,u,\xi,\rho) $ extends to $  \Omega_{\gamma \rho} $ as a holomorphic function $ \widetilde{U}_{\rho}(z,w,\zeta)$ and we have the desired estimate.
  \end{proof} 
   
In using  the estimate \eqref{gesti} we have neglected the Gaussian factor $ e^{-\frac{1}{2} (1-2\gamma-\gamma^2) |\lambda|\rho^2}$ when proving Theorem \ref{thm:solut}. By keeping track of these factors we can actually prove the holomorphic extension of the solution under weaker assumptions on $ f.$ With this in mind, let us define
 \begin{equation}  
 \label{weight2}
 w_k^\lambda(\rho,r)  =\rho^{4s} (c_{k,\rho^2}^{\lambda/4}(s))^2 \psi_k^\lambda(r).
 \end{equation}  As in the proof of Theorem \ref{entire}, we obtain the following identity  valid for any $ 0 < r \leq \gamma \rho $:
$$
 \int_{\He^n} \Big(  \int_{|b| =r} |\widetilde{U}_\rho(a+ib)|^2   d\sigma_r(b)  \Big) da \\ = C_{n,s}  \int_{-\infty}^\infty  \Big( \sum_{k=0}^\infty \|f^\lambda \ast_\lambda \varphi_k^\lambda\|_2^2 \,\, w_k^\lambda(\rho,r) \Big) d\mu(\lambda).
$$
This suggests the introduction of a new space. 
\begin{defn}
For $s, \gamma>0$, we define the space $ \mathcal{H}^s_{\gamma, \rho}(\He^n) $ as the completion of $ C_0^\infty(\He^n) $ with respect to the norm
  $$ 
  \| f \|_{ \mathcal{H}^s_{\gamma,\rho}}^2  =    \int_{-\infty}^\infty  \Big( \sum_{k=0}^\infty \|f^\lambda \ast_\lambda \varphi_k^\lambda\|_2^2 \,\, w_k^\lambda(\rho,\gamma \rho) \Big) d\mu(\lambda).
  $$
  \end{defn}
We also introduce  the following  Hardy type space. 
\begin{defn}
For $s, \gamma >0$, we define the space
$ \widetilde{H}^2(\Omega_{\gamma \rho}) $ consisting of holomorphic functions on $ \Omega_{\gamma \rho} $ for which 
$$  
\| \widetilde{F}\|_{\widetilde{H}^2(\Omega_{\gamma \rho})}^2 =\sup_{0< r < \gamma \rho} \int_{\He^n} \Big(  \int_{|b| =r} |\widetilde{F}(a+ib)|^2  d\sigma_r(b)\Big) da < \infty.
$$
\end{defn}
In view of \eqref{ident}, Theorem \ref{thm:solut} and the fact that $ \psi_k^\lambda(r) $ is an increasing function of $ r $ we observe that, for solutions of the extension problem,
$$  
\| \widetilde{U}_\rho\|_{\widetilde{H}^2(\Omega_{\gamma \rho})}^2 =\int_{\He^n} \Big(  \int_{|b| = \gamma \rho} |\widetilde{U}_\rho(a+ib)|^2\,  d\sigma_{\gamma \rho}(b)\Big) da
$$
provided $ 0 < \gamma < -1+\sqrt{2}$.
With these notations, the proof of the Theorem \ref{thm:solut} leads to the following result, which is the analogue of Proposition \ref{unitary}.
\begin{thm}
For $ 0 < s \le 1/2, $ let  $ U(\cdot,\rho) =  \rho^{2s} f \ast \Phi_{s,\rho}(\cdot)$ where $ f \in  \mathcal{H}^s_{\gamma,\rho}(\He^n) .$ Then for any $ 0 < \gamma \leq -1+\sqrt{2} $ the solution of the extension problem \eqref{epg} extends to $ \Omega_{ \gamma \rho} $ as a holomorphic function $ \widetilde{U}(\cdot,\rho)$, belongs to $\widetilde{H}^2(\Omega_{\gamma \rho}) $ and satisfies the estimate  
$\| \widetilde{U}(\cdot,\rho) \|_{\widetilde{H}^2(\Omega_{ \gamma \rho}) } = C_s \| f \|_{ \mathcal{H}^s_{\gamma,\rho}(\He^n)} $ for all $ \rho >0 .$ Moreover, the map $ T_\rho :  \mathcal{H}^s_{\gamma,\rho}(\He^n) \rightarrow \widetilde{H}^2(\Omega_{ \gamma \rho})$ taking $ f $ into $ \widetilde{U}(\cdot,\rho)$ is surjective.
\end{thm}
\begin{proof} We only need to prove the surjectivity of the map $ T_\rho.$ As it is a constant multiple of an isometry it has closed range. Hence it is enough to show that the range is dense.
By polarising \eqref{ident} we obtain, for $ F, G \in \widetilde{H}^2(\Omega_{ \gamma \rho})$
$$
 \int_{\He^n} \Big(  \int_{|b| =r} F(a+ib)\overline{G(a+ib)}  d\sigma_r(b)  \Big) \,da  = C_{n}  \int_{-\infty}^\infty  \Big( \sum_{k=0}^\infty \langle f^\lambda \ast_\lambda \varphi_k^\lambda,   g^\lambda \ast_\lambda \varphi_k^\lambda \rangle_{L^2(\C^n)} w_k^\lambda(r) \Big)\, d\mu(\lambda)
$$
where $ f $ and $ g$ are the restrictions of $ F $ and $ G $ respectively to $ \He^n.$ If $ G $ is orthogonal to the range of $ T_\rho $ then for any $ f \in \mathcal{H}^s_{\gamma,\rho}(\He^n) $ we have
$$\int_{-\infty}^\infty  \Big( \sum_{k=0}^\infty \langle f^\lambda \ast_\lambda \varphi_k^\lambda,   g^\lambda \ast_\lambda \varphi_k^\lambda \rangle_{L^2(\C^n)} w_k^\lambda(\rho,r) \Big)\, d\mu(\lambda) = 0
$$
for all $ 0 < r \leq \gamma \rho$, where $w_k^{\lambda}(\rho,r)$ is the weight in \eqref{weight2}. As $ f $ is arbitrary, this forces $ g $ to vanish on $ \He^n $ and hence $ G = 0 .$ This proves the density of the range.
\end{proof}

We can now  state and prove the following analogue of Theorem \ref{thm:solepR} for solutions of the extension problem on the Heisenberg group.
\begin{thm} 
\label{charH}
A solution  of the extension problem  is of the form $ U(x,u,\xi,\rho) =  \rho^{2s} f \ast  \Phi_{s,\rho}(x,u,\xi)$ for some  $ f \in L^2(\He^n) $ if and only if there exists $ \gamma> 0 $ such that for each $ \rho > 0$, $U(\cdot,\rho) $
extends to $ \Omega_{\gamma \rho} $ as a holomorphic function $ \widetilde{U}(\cdot,\rho)$, belongs to $ \widetilde{\mathcal{H}}(\Omega_{\gamma \rho}) $ and satisfies the uniform estimate  $\| \widetilde{U}(\cdot,\rho) \|_{\widetilde{\mathcal{H}}(\Omega_{\gamma \rho})} \leq C \rho^{Q/2}\|f\|_2$ for all $ \rho >0 $. 
\end{thm}
\begin{proof} The direct part of the theorem follows from Theorem \ref{thm:solut} in view of the estimates on $ w_k^\lambda(\rho,\gamma \rho)$ in \eqref{maxim}. To prove the converse, let $ U$ be a solution of the extension problem which has a holomorphic extension $ \widetilde{U} $ satisfying the hypothesis of the theorem. Then, for any fixed $ \rho_0$ and any $ 0 <  r \leq \gamma \rho_0$, $0 < \rho \leq \rho_0$, we have
$$ 
 \| \widetilde{U}_\rho \|^2_{\widetilde{\mathcal{H}}(\Omega_r)}  =  C_s \int_{-\infty}^\infty  \Big( \sum_{k=0}^\infty \|U_\rho^\lambda \ast_\lambda \varphi_k^\lambda\|_2^2  \Psi_k^\lambda(r) \Big)\, d\mu(\lambda)\leq C \rho_0^{Q}\|f\|_2^2.
$$
Thus in view of Theorem \ref{thm:general} we can conclude that there exists a subsequence $ \rho_k \rightarrow 0 $ and $ f_{\rho_0} \in \mathcal{H}_r(\He^n) $ such that $ U(\cdot, \rho_k) \rightarrow f_{\rho_0} $ in $ \mathcal{H}_r(\He^n) .$ As $ \mathcal{H}_{\rho_1} (\He^n) \subset \mathcal{H}_{\rho_0}(\He^n) $ for $ \rho_0 < \rho_1$, it follows that $ f $ is independent of $ \rho_0$. The proof will be complete if we show that $ U(\cdot,\rho) = \rho^{2s} f \ast \Phi_{s,\rho}(\cdot)$ is the unique solution of the extension problem.

Thus we have to prove the following uniqueness result for solutions of the extension problem: Suppose $ U_\rho(\cdot) := U(\cdot,\rho) $ satisfies the equation
 $$
 \big(-\mathcal{L}+\partial_{\rho}^2+\frac{1-2s}{\rho}\partial_{\rho}+\frac14\rho^2\partial_{\xi}^2\big)U_{\rho}(x,u,\xi) =0, \quad \text{ in }\,\, \He^n\times \R_+, 
 $$ 
 with 
 initial condition $ U(\cdot,\rho) \rightarrow 0 $ as $ \rho \rightarrow 0$. If we  further assume that $ U_\rho $ has a holomorphic extension to $ \Omega_{\gamma \rho} $ and satisfies the uniform estimates $\| \widetilde{U}(\cdot,\rho) \|_{\widetilde{\mathcal{H}}(\Omega_{\gamma \rho})} \leq C \rho^{Q/2}\|f\|_2$ for all $ \rho >0 ,$ then $ U =0.$
 
 We now proceed to prove this claim. By expanding $U_\rho(x,u,\xi)$ as 
$$
U_\rho(x,u,\xi)= \int_{-\infty}^{\infty}e^{-i\lambda \xi}  \big( \sum_{k=0}^{\infty}U_\rho^{\lambda}\ast_{\lambda}\varphi_k^{\lambda}(x,u) \big) \,d\mu(\lambda)
$$
and making use of the fact  that for any $ f \in L^2(\He^n) $
$$
\mathcal{L}(e^{-i\lambda \xi}f^{\lambda}\ast_{\lambda}\varphi_k^{\lambda})=((2k+n)|\lambda|)  e^{-i\lambda \xi}f^{\lambda}\ast_{\lambda}\varphi_k^{\lambda},
$$
we see that, for any $ g \in L^2(\He^n)$ and $\lambda \neq 0$  the function  $\psi_{k,\lambda}(\rho):= \langle U_\rho^{\lambda}\ast_{\lambda}\varphi_{k}^{\lambda}, g^{\lambda}\ast_{\lambda}\varphi_{k}^{\lambda} \rangle $ is a solution to the ODE
$$
\begin{cases}
\big(-(2k+n)|\lambda|+\partial_{\rho}^2+\frac{(1-s)}{\rho}\partial_{\rho}-\frac14 \rho^2 \lambda^2\Big)\psi_{k,\lambda}(\rho)=0,\quad \rho>0,\\
\psi_{k,\lambda}(0)= 0.
\end{cases}
$$
We also observe that for any $ \varphi \in L^2(\R) $ we have the uniform estimate
\begin{equation*}  
\Big| \int_{-\infty}^\infty \psi_{k,\lambda}(\rho) \varphi(\lambda) d\mu(\lambda) \Big| \leq \| \varphi \|_{L^2(\R)}  \|g\|_{L^2(\He^n)} \|U_\rho\|_{L^2(\He^n)} \leq C  \rho^{Q/2}\|f\|_2.
\end{equation*}
 
It is enough to show that $ \psi_{k,\lambda}(\rho) = 0 $ for all $k \in \mathbb{N}$,  $\lambda \neq 0$ and $\rho>0.$
Observe that $\phi_{k,\lambda}(\rho)=\psi_{k,\lambda}(\sqrt{2\rho})$ solves
$$
\begin{cases}
\big(-(2k+n)|\lambda|+\rho\partial_{\rho}^2+(1-s)\partial_{\rho}-\rho\lambda^2\Big)\phi_{k,\lambda}(\rho)=0, \quad \rho>0,\\
\phi_{k,\lambda}(0)=0.
\end{cases}
$$
We are reduced to find a solution $\phi_{k,\lambda}(\rho)$ to the above ODE in the variable $\rho$, such that $\phi_{k,\lambda}(\rho)\to 0$ as $\rho\to 0$ 
and satisfying the estimate
\begin{equation}  
\label{estimphi}
\Big| \int_{-\infty}^\infty \phi_{k,\lambda}(\rho) \varphi(\lambda) d\mu(\lambda) \Big|  \leq C \rho^{Q/2}.
\end{equation}
By taking $\phi_{k,\lambda}(\rho)=e^{-|\lambda|\rho}g_k(2|\lambda|\rho)$, the equation for $\phi_{k,\lambda}$ is equivalent to the following equation for $g_k(r)$, with $r:=2|\lambda|\rho$
$$
rg_k^{''}(r)+(1-s-r)g_k'(r)-\frac{2k+n+1-s}{2}g_k(r)=0
$$
and another transform $g_k(r)=r^sh_k(r)$ leads to
\begin{equation}
\label{kummer}
rh_k^{''}(r)+(1+s-r)h_k'(r)-\frac{2k+n+1+s}{2}h_k(r)=0.
\end{equation}
The boundary condition becomes
\begin{equation}
\label{bry}
\lim_{\rho\to 0}\phi_{k,\lambda}(\rho)=\lim_{\rho\to 0}e^{-|\lambda|\rho}g_k(2|\lambda|\rho)=\lim_{\rho\to 0}e^{-|\lambda|\rho}(2|\lambda|\rho)^sh_k(2|\lambda|\rho)= 0
\end{equation}
and the uniform estimate \eqref{estimphi} reads as
 \begin{equation}  
 \label{estimateh}
 \Big| \int_{-\infty}^\infty e^{-|\lambda|\rho}(2|\lambda|\rho)^sh_k(2|\lambda|\rho) \varphi(\lambda) d\mu(\lambda) \Big|  \leq C \rho^{Q/2}.
\end{equation}

Equation \eqref{kummer} is a Kummer's equation and it has two linearly independent functions $ M(a,b,r) $ and $ V(a,b,r) $ where $ a = \frac{1}{2}(2k+n+1+s) $ and $ b = 1+s$, see \cite[Chapter 13]{AS} and also \cite[Lemma 5.2]{FGMT}. Thus
$$ 
h_k(r) = C_1 V(a,b,r) +C_2 M(a,b,r).
$$ 
As $ r \rightarrow \infty $ the function $ M $ has the asymptotic property
$$ 
M(a,b,r) = \frac{\Gamma(b)}{\Gamma(a)} e^r r^{a-b} (1+O(r^{-1})) 
$$ 
which along with  the condition  \eqref{estimateh} forces $ C_2 =0.$  As $ r \rightarrow 0 $ we have 
$$ 
r^{-1+b} V(a,b,r) =  \frac{\Gamma(b-1)}{\Gamma(a)}+o(1) 
$$
which forces $ C_1 = 0 $ because of the initial condition \eqref{bry}.  Thus $ h_k(r) = 0 $. This completes the proof of the uniqueness and hence the theorem is completely proved.
\end{proof}

 \subsection{Holomorphic extensions of eigenfunctions of $ \Delta_S$} 
 
 Using the connection between eigenfunctions $ \widetilde{W} $ of $ \Delta_S $ and solutions $ U $ of the extension problem in Proposition \ref{conne}, we can deduce holomorphic properties of $ \widetilde{W} $ from  Theorem \ref{charH}. Thus we see that for any $ f \in L^2(\He^n) $ the function
 $$
  \widetilde{W}(z,w,\zeta,\rho) =  2^s \rho^{\frac{n+1+s}{2}} f \ast \Phi_{s, \sqrt{2\rho}}(2^{-1/2}z,2^{-1/2}w,2^{-1}\zeta) 
  $$
 is holomorphic on the domain $ \Omega_{2\gamma \sqrt{\rho}}$,  belongs to $ \widetilde{\mathcal{H}}(\Omega_{2\gamma \sqrt{\rho}}) $  and satisfies the uniform estimate
 $$
 \| \widetilde{W}(\cdot,\rho) \|_{\widetilde{\mathcal{H}}(\Omega_{2\gamma \sqrt{\rho}})} \leq C_s \rho^{(Q-s)/2}  \|f\|_2.
 $$
 As in the case of the real hyperbolic space we can restate the above in terms of  Poisson integrals on the solvable group $ S.$ Let us recall here the functions \eqref{varph} and \eqref{Phisr}. By defining
 $$ 
 \varphi_{s,\delta}(x,u,\xi) = \big((\delta+\frac{1}{4}|(x,u)|^2)^2+\xi^2\big)^{-\frac{n+1+s}{2}}
 $$ 
we observe that $ \Phi_{s,\rho} $ can be written in terms of $ \varphi_{s,\delta} $ as
\begin{equation*} 
\Phi_{s,\rho}(x,u,\xi) = \frac{2^{-(n+1+s)} }{ \pi^{n+1}\Gamma(s)}  \Gamma\Big(\frac{n+1+s}{2}\Big)^2 \varphi_{s,\delta}(x,u,\xi)
 \end{equation*}
 with $ \delta = \frac{1}{4}\rho^2$. In other words, 
 \begin{equation}
 \label{relH}
 \Phi_{s,\sqrt{2\rho}}(2^{-1/2}x,2^{-1/2}u,2^{-1}\xi) = \frac{\Gamma(\frac{n+1+s}{2})^2 }{ \pi^{n+1}\Gamma(s)}   \varphi_{s,\rho}(x,u,\xi).
 \end{equation}
Now we define $ \widetilde{f}(x,u,\xi) =  f(2^{-1/2}x,2^{-1/2}u,2^{-1}\xi)$, an easy calculation shows that
\begin{equation*}
 \widetilde{W}(x,u,\xi,\rho) = 2^{-(n+1-s)} \rho^{\frac{n+1+s}{2}}\frac{\Gamma(\frac{n+1+s}{2})^2 }{ \pi^{n+1}\Gamma(s)}  \widetilde{f}\ast \varphi_{s,\rho}(x,u,\xi).
\end{equation*}

The kernels $   \varphi_{i\lambda,\rho}(x,u,\xi)$ defined for $ \lambda \in \C $ are called generalised Poisson kernels and they occur in the definition of the Helgason Fourier transform on the solvable group $ S.$ For $ f \in L^1(S),$ its Fourier transform is defined by  
$$ 
\widehat{f}(\lambda,(x,u,\xi)) = \int_0^\infty  f(\cdot,\rho)\ast \varphi_{i\lambda,\rho}(x,u,\xi) \rho^{\frac{n+1+i\lambda}{2}} \rho^{-n-2} \,d\rho
$$
for $ \lambda \in \R$, $(x,u,\xi) \in \He^n$, we refer to Astengo et al. \cite{ACDb} for a comprehensible study of the Fourier transform on solvable extensions of $H$-type groups.  The inversion formula is given by
 \begin{equation}
 \label{invH}
 f(x,u,\xi,\rho) = C \int_{-\infty}^\infty  \rho^{\frac{n+1-i\lambda}{2}} \widehat{f}(\lambda,\cdot)\ast \varphi_{-i\lambda,\rho}(x,u,\xi) |c(\lambda)|^2 d\lambda
 \end{equation} 
 where $ c(\lambda) $ is the Harish-Chandra's c-function. In view of the fact that $ \widehat{f}(\lambda,(x,u,\xi))$  does not reduce to the spherical Fourier transform of $ f $ when $ f $ is radial, it is convenient to work with the normalised Helgason Fourier transform defined by
 $$ 
 \mathcal{H}f(\lambda,(x,u,\xi)) =  \big((1+\frac{1}{4}|(x,u)|^2)^2+\xi^2\big)^{- \frac{n+1+i\lambda}{2}}   \widehat{f}(\lambda,(x,u,\xi)) .
 $$
 Then the inversion formula \eqref{invH} takes the following form:
 $$ 
 f(x,u,\xi,\rho) = C \int_{-\infty}^\infty  \rho^{\frac{n+1-i\lambda}{2}} \big( \varphi_{i\lambda,1}(\cdot)\mathcal{H}f(\lambda,\cdot)\big)\ast \varphi_{-i\lambda,\rho}(x,u,\xi) |c(\lambda)|^2\, d\lambda.
 $$ 
 
  This suggests that the right Poisson transform for the complex hyperbolic space viewed as the  solvable extension $ S $ of the Heisenberg group is given by
 \begin{equation} 
 \label{Poisson}
 \mathcal{P}_\lambda f(x,u,\xi,\rho)  =\rho^{\frac{n+1-i\lambda}{2}}  \big( \varphi_{i\lambda,1}(\cdot)f(\cdot)\big)\ast \varphi_{-i\lambda,\rho}(x,u,\xi) 
 \end{equation}
 where $ f $ is a function on $ \He^n .$   
 From Theorem \ref{charH} we can deduce as an immediate corollary the following characterisation of certain eigenfunctions of $ \Delta_S $ expressible as Poisson transforms of functions on $ \He^n$. 
 
  \begin{thm}
  \label{charE}  
  Let $ 0 < s < 1.$ An eigenfunction $ \widetilde{W} $ of the Laplace-Beltrami operator $ \Delta_S $ on $ S $ with eigenvalue $ -\frac{1}{4}((n+1)^2-s^2) $  can be expressed as the Poisson integral $ \mathcal{P}_{is}f $ with $ f \in L^2(\He^n, (\varphi_{0,1}(h))^2 dh) $ if and only if  there exists a $ \gamma> 0 $ such that  $ \widetilde{W}(\cdot,\rho)  \in \widetilde{\mathcal{H}}(\Omega_{\gamma \sqrt{\rho}}) $ and satisfies the uniform estimate  $\| \widetilde{W}(\cdot,\rho) \|_{\widetilde{\mathcal{H}}(\Omega_{\gamma \sqrt{\rho}})} \leq C
  \rho^{(Q-s)/2}$ for all $ \rho >0 .$

 \end{thm} 
 
 As a final remark, we observe that the Poisson transform $ f \rightarrow   \mathcal{P}_\lambda f $  takes functions (distributions) on $ \He^n $ into eigenfunctions of the Laplace-Beltrami operator $ \Delta_S $ with eigenvalues $ -\frac{1}{4}((n+1)^2+\lambda^2).$ Adapted to our setting, the celebrated Helgason conjecture addresses the converse: is every eigenfuction of $ \Delta_S $ the Poisson transform of a distribution on $ \He^n$? This conjecture has been solved for all Riemannian symmetric spaces of non-compact type in the compact picture $ X = G/K.$ However, in the context of general $ NA $ groups, the conjecture is still open.
 
 \section{Extension to $H$-type groups} 
 \label{extensionH}
 In this section  extend the results in Theorems \ref{charH} and \ref{charE} to $H$-type groups. The strategy  consists of using partial Radon transform in the central variable, reducing the problem to the case of Heisenberg group, and deducing results on the general $ NA $ groups from the results of previous section.

\subsection{$H$-type Lie algebras and groups}
\label{subsec:Htype}
 A step two nilpotent Lie group $N$ is said to be an $H$-type group if its Lie algebra $\mathfrak{n}$ is of $H$-type. A Lie algebra $\mathfrak{n}$ is said to be an $H$-type Lie algebra if we can write $\mathfrak{n}$ as the direct sum $\mathfrak{v}\oplus \mathfrak{z}$ of two Euclidean spaces with a Lie algebra structure such that $\mathfrak{z}$ is the centre of $\mathfrak{n}$ and for every unit vector $ v \in\mathfrak{v}$ the map $\operatorname{ad}(v)$ is a surjective isometry of the orthogonal complement of $\operatorname{ker}(\operatorname{ad}(v))$ onto $\mathfrak{v}$. If $\mathfrak{n}$ is such an $H$-type algebra we define a map $J:\mathfrak{z}\to \operatorname{End}(\mathfrak{v})$ by
$$
(J_\omega v,v')=(\omega,[v,v']), \qquad \omega\in \mathfrak{z}, \quad v,v'\in \mathfrak{v}.
$$ 
It then follows that $J_\omega^2=-I$ whenever $\omega$ is a unit vector in $\mathfrak{z}$. We can therefore introduce a complex structure on $\mathfrak{v}$ using $J_\omega$. The Hermitian inner product on $\mathfrak{v}$ is given by 
$$
\langle v,v'\rangle_\omega=(v,v')+i(J_\omega v,v')=(v,\omega)+i([v,v'],\omega).
$$
Thus when $N$ is an $H$-type group, identifying $N$ with its Lie algebra $\mathfrak{n}$, we write the elements of $N$ as $(v,t)$, $v\in \mathfrak{v}$, $t\in \mathfrak{z}$. In view of the Baker--Campbell--Hausdorff formula, the group law takes the form
$$
(v,t)(v',t')=(v,t)+(v',t')+\frac12[(v,t),(v',t')].
$$
The best known example of an $H$-type group is the Heisenberg group $\He^n=\R^{2n}\times \R$. 

The Heisenberg groups play an important role in studying problems on $H$-type groups. This is due to the fact that to every $H$-type Lie algebra $\mathfrak{n}=\mathfrak{v}\oplus \mathfrak{z}$ and unit vector $\omega\in \mathfrak{z}$ we can associate a Heisenberg Lie algebra $\mathfrak{h}_\omega$ as follows. Given a unit vector $\omega\in \mathfrak{z}$, let $k(\omega)$ stand for the orthogonal complement of $\omega$ in $\mathfrak{z}$. Then the quotient algebra $\mathfrak{n}(\omega)=\mathfrak{n}/k(\omega)$ can be identified with $\mathfrak{v}\oplus \R$ by defining
$$
[(v,\xi),(v',\xi')]_\omega=(0,[J_\omega v,v']).
$$
It is known (see \cite{KR,MuSe}) that this algebra is isomorphic to the Heisenberg algebra $\mathfrak{h}^n$. We denote the corresponding group by $\He_\omega^n$, of dimension $(2n+1)$, which is isomorphic to $\He^n$.

 $H$-type groups were first introduced in \cite{K}. A full discussion and more examples of $H$-type groups can be found in \cite[Chapter 18]{BLU} and \cite{KR}.

\subsection{The representation theory of $H$-type groups}
\label{subsec:reprH}

Before describing the representation theory of $H$-type groups, let us first recall some facts about irreducible unitary representations of the Heisenberg groups $\He^n$. It is well known that any irreducible unitary representation of $\He^n$ which is nontrivial at the centre (namely on $\{0\}\times \R$) is unitarily equivalent to the Schr\"odinger representation $\pi_{\lambda}$, for a unique $\lambda\in \R^{*}=\R\setminus\{0\}$. Here these representations $\pi_{\lambda}$ are all realised on $L^2(\R^n)$ and given explicitly by 
$$
\pi_{\lambda}(z,\xi)\varphi(\eta)=e^{i\lambda \xi}e^{i(x\cdot \eta+\frac12 x\cdot u)}\varphi(\eta+y)
$$
where $z=x+iu$, $\varphi\in L^2(\R^n)$. There is another family of one dimensional representations which do not play any role in the Plancherel theorem. Hence we do not attempt to describe them.

The group Fourier transform of an $L^1(\He^n)$ function $f$ is defined to be the operator valued function $\lambda\to \widehat{f}(\lambda)$ given by 
$$
\widehat{f}(\lambda)=\int_{\He^n} f(z,\xi)\pi_{\lambda}(z,\xi)\,dz\,d\xi.
$$
Sometimes we use the notation $\pi_{\lambda}(f)$ instead of $\widehat{f}(\lambda)$. Recalling the definition of $\pi_{\lambda}$ it is easy to see that 
$$
\widehat{f}(\lambda)=\int_{\C^n} f^{\lambda}(z)\pi_{\lambda}(z,0)\,dz
$$
where $f^{\lambda}$ was defined in \eqref{eq:inverseFT}. We will be using this notation without any further comments.

When $f\in L^1\cap L^2(\He^n)$ it can be easily verified that $\widehat{f}(\lambda)$ is a Hilbert--Schmidt operator and we have 
$$
\int_{\He^n}|f(z,\xi)|^2\,dz\,d\xi=(2\pi)^{-n-1}\int_{-\infty}^{\infty}\|\widehat{f}(\lambda)\|_{\operatorname{HS}}^2|\lambda|^n\,d\lambda.
$$
The above equality of norms allows us to extend the definition of the Fourier transform to all $L^2$ functions. It then follows that we have Plancherel theorem: $f\to \widehat{f}$ is a unitary operator from $L^2(\He^n)$ onto $L^2(\R^*, \textrm{S}_2, d\mu)$ where $ \textrm{S}_2$ stands for the space of all Hilbert--Schmidt operators on $L^2(\R^n)$ and $d\mu(\lambda)=(2\pi)^{-n-1}|\lambda|^nd\lambda$ is the Plancherel measure for the group $\He^n$.

The connection between $H$-type Lie algebras and Heisenberg Lie algebras allows us to get a quick picture of the representation theory of $H$-type groups. As in the case of the Heisenberg groups, the irreducible unitary representations of $H$-type group $N$ comes in two groups. As before we neglect the one dimensional representations which are trivial on the centre of $N$. If $\pi$ is any infinite dimensional irreducible representation of $N$, then its restriction to the centre has to be a unitary character. This means that $\exists \lambda\in \R^*$ and $\omega\in \mathbb{S}^{m-1}$, the unit sphere in the centre (identified with $\R^m$) such that $\pi(0,t)=e^{i\lambda \omega \cdot t}\operatorname{Id}$. It can be shown that such a representation factors through a representation of $\He_{\omega}^n$, the group introduced in Subsection \ref{subsec:Htype}. By making use of the Stone--von Neumann theorem we can show that all infinite-dimensional irreducible unitary representations of $N$ are parametrised by $(\lambda,\omega)$, $\lambda>0$, $\omega\in \mathbb{S}^{m-1}$. We denote such a representation by $\pi_{\lambda,\omega}$. It follows that the restriction of $\pi_{\lambda,\omega}$ to $\He_{\omega}^n$ is unitarily equivalent to the Schr\"odinger representation $\pi_{\lambda}$.

\subsection{The Radon transform}
\label{subsec:radon}

In order to study the extension problem for the sublaplacian  on $ N $ we will make use of the (partial) Radon transform in the central variable. It is therefore helpful to collect some results on the Radon transform of a function $ f $ on $ \R^m.$ Given $ f \in L^1(\R^m) $ its Radon transform is a function  on $ \R \times \mathbb{S}^{m-1} $ defined by
$$ 
Rf(\xi,\omega) = \int_{ y \cdot \omega = \xi} f(y) \, dy = \int_{k(\omega)} f(\xi \omega +\eta)\, d\eta, 
$$ 
where $ k(\omega) $ is the orthogonal complement of $ \omega $ and $ d\eta $ is the $(m-1)$ dimensional Lebesgue measure. We immediately see that 
\begin{equation}
\label{wideh} 
 \widehat{f}(\lambda \omega) = \int_{-\infty}^\infty Rf(\xi, \omega) e^{-i\lambda \xi} d\xi,
 \end{equation}
 where here $\widehat{f}$ denotes the Fourier transform in the $t$ variable.
 Consequently, the Radon transform of $ f $ satisfies the equation 
$$ 
\int_{-\infty}^\infty Rf(\xi,\omega) e^{-i\lambda \xi} d\xi = \int_{-\infty}^\infty Rf(\xi, -\omega) e^{i\lambda \xi} d\xi.
$$  
When $ f \in L^1 \cap L^2(\R^m) $ its Radon transform $ Rf(\xi, \omega) $ needs not be in $ L^2(\R \times \mathbb{S}^{m-1})$.  However, it is easily seen that the function $ R ((-\Delta)^{(m-1)/4}f)(\xi,\omega) $ belongs to $ L^2(\R \times \mathbb{S}^{m-1})$. This follows from the Plancherel theorem for the Fourier transform and the relation \eqref{wideh}. We will call $ R_{\operatorname{mod}} := R (-\Delta)^{(m-1)/4}$ the \textit{modified Radon transform}. For any $ F \in  L^2(\R \times \mathbb{S}^{m-1}) $ we let $ \widetilde{F}(\lambda,\omega) $ stand for the Fourier transform of $ F $ in the $ \xi $ variable.

\begin{prop}
\label{prop:coco}
The image of $ L^2(\R^m) $ under $ R_{\operatorname{mod}}$  is a closed subspace of $ L^2(\R \times \mathbb{S}^{m-1}).$  An element $ F $ of $ L^2(\R \times \mathbb{S}^{m-1})$ belongs to this subspace if and only if it satisfies the \textit{compatibility condition} 
\begin{equation}
\label{coco}
\widetilde{F}(\lambda, \omega) =\widetilde{F}(-\lambda, -\omega).
\end{equation}
\end{prop}
\begin{proof} When $ F = R_{\operatorname{mod}}f$, for $f \in L^2(\R^m) $, it is clear that $ \widetilde{F}(\lambda, \omega) =\widetilde{F}(-\lambda, -\omega).$
It is also easy to verify that 
$$  
\int_{\R^m}  |f(t)|^2\, dt = c \int_0^\infty \int_{\mathbb{S}^{m-1}} |R_{\operatorname{mod}}f(\xi,\omega)|^2 \,d\sigma(\omega)\, d\lambda.
$$
The  relation \eqref{wideh} allows us to reconstruct $ f $ from $ Rf(\xi,\omega) $ leading to an inversion formula for the Radon transform: 
\begin{equation*}
f(t) = (2\pi)^{-m} \int_0^\infty \int_{\mathbb{S}^{m-1}} \widetilde{Rf}(\lambda, \omega) \lambda^{m-1} \,d\sigma(\omega) \,d\lambda.
\end{equation*}

Given $ F \in L^2(\R \times \mathbb{S}^{m-1}) $ which satisfies the compatibility condition \eqref{coco},
let us define $ f $ on $ \R^m $ by  the prescription
\begin{equation}   
\label{presc}
f(t) = (2\pi)^{-m} \int_0^\infty \int_{\mathbb{S}^{m-1}} e^{i \lambda t \cdot \omega} \widetilde{F}(\lambda,\omega) \lambda^{(m-1)/2}\, d\sigma(\omega)\, d\lambda.
\end{equation} 
Then it follows that $ \widehat{f}(\lambda \omega) = \widetilde{F}(\lambda,\omega) |\lambda|^{(1-m)/2} $ for any $  (\lambda,\omega) \in \R \times \mathbb{S}^{m-1}.$ Consequently, we have
\begin{equation*}  
R ((-\Delta)^{(m-1)/4}f)(\xi,\omega) = F(\xi,\omega).
\end{equation*} 
Moreover, we also observe that
$$ 
\int_0^\infty \int_{\mathbb{S}^{m-1}} |\widehat{f}(\lambda \omega)|^2 \lambda^{m-1} \,d\sigma(\omega)\, d\lambda =  \int_0^\infty \int_{\mathbb{S}^{m-1}} |\widetilde{F}(\lambda,\omega)|^2 \,d\sigma(\omega)\, d\lambda 
$$
which simply means that $ f \in L^2(\R^m).$ Thus formula \eqref{presc} provides us with an inversion formula for the modified Radon transform. 

Finally, we can show that the image is a closed subspace. To see this, let $ F_n = R_{\operatorname{mod}}f_n $ converge to $ F $ in $L^2(\R \times \mathbb{S}^{m-1}) .$  In view of the inversion formula, we only need to verify the compatibility condition \eqref{coco}. Under the assumption, $ \widetilde{F_n}(\lambda,\omega) $ converges to $ \widetilde{F}(\lambda,\omega)$ in $ L^2(\R \times \mathbb{S}^{m-1})$. Since $ \widetilde{F_n}(\lambda,\omega) = |\lambda|^{(m-1)/2} \widehat{f_n}(\lambda \omega)  = \widetilde{F_n}(-\lambda,-\omega) ,$ it also converges to $ \widetilde{F}(-\lambda,-\omega)$. Hence the compatibility condition \eqref{coco}.
\end{proof}

In dealing with functions on the $H$-type group $ N $, we will use specific notations for the different Radon transforms introduced above. Given an integrable function $f$ on $N$ and  $\omega\in \mathbb{S}^{m-1}$, we define the (partial) Radon transform by 
$$
f_\omega(x,u,\xi)=\int_{y\cdot \omega =\xi}f(x,u,y)\,dy=\int_{k(\omega)}f(x,u,\xi\omega+\eta)\,d\eta,\qquad x,u\in \R^n, \quad \xi\in \R,
$$
where $d\eta$ is the $(m-1)$-dimensional Lebesgue measure on $k(\omega)$. 
From the definition, it follows that
$$ 
\int_{\R^m} e^{-i \lambda \omega \cdot t} f(x,u,t) dt =  \int_{-\infty}^\infty  e^{-i \lambda \xi} f_\omega(x,u,\xi) \,d\xi.
$$
The compatibility condition reads as follows: the above equation gives  
\begin{equation}
\label{cocoH} 
\widetilde{f_\omega}(x,u,\lambda) = \widetilde{f_{-\omega}}(x,u,-\lambda) \end{equation}
where $  \widetilde{f_\omega}(x,u,\lambda) $ stands for the Fourier transform of $ f_\omega(x,u,\xi) $ in the $\xi$ variable. Arguing as above, we see that if we have a family of functions $ F_\omega(x,u,\xi) $ indexed by $ \omega \in \mathbb{S}^{m-1} $   satisfying the compatibility condition 
$ \widetilde{F_\omega}(x,u,\lambda) = \widetilde{F_{-\omega}}(x,u,-\lambda)$ and for which 
$$  
\int_{\mathbb{S}^{m-1}} \int_{\He^n} |F_\omega(x,u,\xi)|^2 \,dx\, du\, d\xi \,d\sigma(\omega) < \infty ,
$$  
then the function $ f $ defined by
\begin{equation*} 
f(x,u,t) = (2\pi)^{-m}  \int_0^\infty \int_{\mathbb{S}^{m-1}} e^{i \lambda \omega \cdot t} \widetilde{F_\omega}(x,u,\lambda) \lambda^{(m-1)/2} \,d\sigma \,d\lambda 
\end{equation*}
satisfies the relation $ R_{\omega}f := ((-\Delta_t)^{(m-1)/4} f)_\omega= F_\omega$. Observe that here we denote $ R_\omega f $ for the modified partial Radon transform in the central variable.

For each $\omega$, $f_{\omega}$ can be considered as a function on $\He_{\omega}^n$.
The collection $f_{\omega}$ completely determines $f$. Moreover, it can be verified that   
$$
(f\ast g)_{\omega}=f_{\omega}\ast_{\omega}g_{\omega}
$$
for two functions $f,g\in L^1(N)$. In the above, the convolution on the left is on the group $N$ whereas $\ast_{\omega}$ on the right stands for the convolution on the Heisenberg group $\He^n_{\omega}$.
Using the above relation and the connection between $\pi_{\lambda,\omega}$ and $\pi_{\lambda}$ we can show that 
$$
\pi_{\lambda,\omega}(f)=\pi_{\lambda}(f_{\omega}), \qquad \omega\in \mathbb{S}^{m-1}, \quad \lambda>0.
$$
The function $f_{\omega}$ does not belong to $L^2(\He^n_{\omega})$, but the modified Radon transform 
\begin{equation*}
R_\omega f(x,u,\xi) = ((-\Delta_t)^{(m-1)/4} f)_\omega=  |\partial_{\xi}|^{(m-1)/2}f_\omega(x,u,\xi),
\end{equation*}
does  for almost every $ \omega$.  This  is a consequence of the easily verified relation 
 $$ 
 \int_{\R^m} |f(x,u,t)|^2 dt =  c \int_{\mathbb{S}^{m-1}}\int_0^\infty  |(\widetilde{R_\omega f})(x,u,\lambda)|^2  d\sigma(\omega)\, d\lambda. 
 $$
From  $ \pi_{\lambda,\omega}(f)=\pi_{\lambda}(f_{\omega})$,  we also obtain $
\pi_{\lambda}(R_{\omega}f)=|\lambda|^{(m-1)/2}\pi_{\lambda}(f_{\omega})$  for $ \omega\in \mathbb{S}^{m-1}$, $\lambda>0$.

\subsection{Holomorphic extensions of solutions of the extension problem in $H$-type groups}

We define the sublaplacian $\mathcal{L}_N$ on a $H$-type group $N$. We fix an orthonormal basis $X_j$, $j=1,2,\ldots,2n$ for the Lie algebra $\mathfrak{v}$ and $Z_j$, $j=1,2,\ldots, m$ for the Lie algebra $\mathfrak{z}$. We denote by $\Delta_t=\sum_{j=1}^mZ_j^2$ the ordinary Laplacian on the centre of $N$ (in $\mathbb{R}^m$). The sublaplacian $\mathcal{L}_N$ on $N$ is defined by $\mathcal{L}_N=-\sum_{j=1}^{2n}X_j^2$.  We consider the extension problem
\begin{equation}
\label{ep}
\big( -\mathcal{L}_N + \partial_\rho^2 +\frac{1-2s}{\rho} \partial_\rho +\frac{1}{4}\rho^2 \Delta_{t} \big) U(x,u,t,\rho) = 0,\qquad  U(x,u,t,0) = f(x,u,t).
\end{equation}
By denoting $L_s:= -\mathcal{L}_N + \partial_\rho^2 +\frac{1-2s}{\rho} \partial_\rho +\frac{1}{4}\rho^2 \Delta_{t} $, we have that $R_{\omega}(L_s U)=\mathcal{L}_s R_{\omega}U$, where $\mathcal{L}_s:= -\mathcal{L} + \partial_\rho^2 +\frac{1-2s}{\rho} \partial_\rho +\frac{1}{4}\rho^2 \partial_{\xi}^2$. This follows from the fact that under the Radon transform $R_{\omega}$, the sublaplacian $\mathcal{L}_N$ on $N$ becomes the sublaplacian $\mathcal{L}$, see \cite[Appendix]{Mu}, also \cite[Section 3]{MuSe}. Then, $R_{\omega}U(x,u,\xi,\rho)$ is a solution to the extension problem
\begin{equation}
\label{epr}
\big( -\mathcal{L} + \partial_\rho^2 +\frac{1-2s}{\rho} \partial_\rho +\frac{1}{4}\rho^2 \partial_{\xi}^2\big)R_{\omega}U(x,u,\xi,\rho)= 0,\qquad  R_{\omega}U(x,u,\xi,0) = R_\omega f(x,u,\xi),
\end{equation}
where $\mathcal{L}$ stands for the sublaplacian on $\He^n$.  
For $(x,u,t)\in N$ and $0<s<1$, let
\begin{equation}
\label{PhiH}
\Phi_{s,\rho}^N(x,u,t) :=\frac{2^{2m+n+s-1}\Gamma(\frac{n+1+s}{2})\Gamma(\frac{n+m+s}{2})}{\pi^{n+(m+1)/2}\Gamma(s)}\big((\rho^2+|x|^2+|u|^2)^2+16|t|^2\big)^{- \frac{n+m+s}{2}} . 
\end{equation}
We simply write $ \Phi_{s,\rho} $ when $ N = \He^n$ (and therefore $m=1$), that is exactly \eqref{Phi}.
We are in position to prove the result analogous to Theorem \ref{charH}. As usual let us employ the notation $ U_\rho(\cdot) := U(\cdot,\rho)$. 
\begin{thm} 
\label{charHtypecor}
Let $m>1$. A solution  of the extension problem  is of the form $ U(x,u,t,\rho) =  \rho^{2s} f \ast  \Phi_{s,\rho}^N(x,u,t)$ for some  $ f \in L^2(N) $ if and only if there exists $ \gamma> 0 $ such that for each $ \rho > 0$, the function $ (x,u,\xi,\omega) \rightarrow R_{\omega}U_\rho(x,u,\xi) $ extends to $ \Omega_{\gamma \rho} \times \mathbb{S}^{m-1} $ as an $ L^2(\mathbb{S}^{m-1})$-valued 
holomorphic function $ \widetilde{R_{\omega}U_\rho}(z,w,\zeta)$ and satisfies the uniform estimate 
$$  
\int_{\Omega_{\gamma \rho}} \Big( \int_{\mathbb{S}^{m-1}} | \widetilde{R_{\omega}U_\rho}(z,w,\zeta) |^2 d\sigma(\omega) \Big) \,dz\, dw\, d \zeta \leq C \rho^{2(n+1)}\|f\|_2^2
$$ 
for all $ \rho >0 $. 
\end{thm} 

\begin{proof}
Analogously as in Theorem \ref{charH}, the direct part follows from Theorem \ref{thm:solut}. Indeed, since the modified Radon transform $R_{\omega}U_\rho(x,u,\xi)$ is a solution to the extension problem \eqref{epr}, Theorem~\ref{thm:solut} implies that $\big\|\|\widetilde{R_{\omega}U_\rho}\|_{L^2(\mathbb{S}^{m-1})}\big\|_{\Omega_{\gamma \rho}}\le C\rho^{(n+1)}\|f\|_2$.

Let us prove the converse.  As $ U $ is a solution of the extension problem for $ \mathcal{L}_N,$ it follows that  for each $\omega \in \mathbb{S}^{m-1}$, the modified Radon transform $R_{\omega}U_\rho(x,u,\xi) $ is a solution to \eqref{epr} which has a holomorphic extension $ \widetilde{R_{\omega}U_\rho}(z,w,\zeta)$  to the domain $ \Omega_{\gamma \rho}$. The hypothesis of the theorem says  that $ \omega \rightarrow  \widetilde{R_{\omega}U_\rho}(\cdot)$ is an $ L^2(\mathbb{S}^{m-1}) $ function taking values in $ \widetilde{\mathcal{H}}(\Omega_{\gamma \rho}) $. From this we need to conclude that there exists an $ L^2(\mathbb{S}^{m-1}) $ function $ g^\omega $ taking values in $ L^2(\He^n) $ such that 
$R_{\omega}U_\rho(x,u,\xi) = \rho^{2s} g^{\omega}\ast_{\omega} \Phi_{s,\rho}(x,u,\xi) $.

To this end, let $ Y_j$, $j \in \mathbb{N} $ be an orthonormal basis for $ L^2(\mathbb{S}^{m-1}) $ and consider 
$$  
F_\rho^j(z,w,\zeta) = \int_{\mathbb{S}^{m-1}}  \widetilde{R_\omega U_\rho}(z,w,\zeta) \overline{Y_j(\omega)} \,d\sigma(\omega).
$$
As $  F_\rho^j(x,u,\xi)$ is a solution of the extension problem on $ \He^n $ having a holomorphic extension to $ \Omega_{\gamma \rho} $  satisfying the required estimates, we can conclude, by Theorem \ref{charH}, that $  F_\rho^j = \rho^{2s} g_j \ast_{\omega} \Phi_{s,\rho}$  for some $ g_j \in L^2(\He^n)$. Now
$$ 
\widetilde{R_\omega U_\rho}(z,w,\zeta) = \sum_{j=0}^\infty  F_\rho^j(z,w,\zeta) Y_j(\omega) 
$$
where the series converges in $ L^2\big(\Omega_{\gamma \rho}, L^2(\mathbb{S}^{m-1})\big)$. From this fact, we then claim that the series 
$ \sum_{j=0}^\infty  g_j(x,u,\xi) Y_j(\omega)$ converges in $  L^2(\He^n, L^2(\mathbb{S}^{m-1}))$. Indeed, we are given that 
$$   
\sum_{j=0}^\infty  \Big( \int_{\Omega_{\gamma \rho}} |F_\rho^j(z,w,\zeta)|^2 \,dz\, dw\, d\zeta \Big) \leq C \rho^{2(n+1)}.
$$
Since $ F_\rho^j = \rho^{2s} g_j \ast_{\omega} \Phi_{s,\rho} ,$ in view of  Theorem \ref{thm:general} we have
\begin{equation} 
\label{ineqF}
  \sum_{j=0}^\infty  \int_{-\infty}^\infty  \Big( \sum_{k=0}^\infty \|  \rho^{2s}  g_j^\lambda \ast_\lambda \Phi_{s,\rho}^\lambda \ast_\lambda \varphi_k^\lambda\|_2^2 \,\, \Psi_k^\lambda(\rho) \Big) d\mu(\lambda) \leq C \rho^{2(n+1)}.
\end{equation}
Observe that $ \rho^{-2(n+1)} \Psi_k^\lambda(\gamma \rho) $ converges to $ C_\gamma  \psi_k^\lambda(0) = C_\gamma $ as $ \rho $ goes to $ 0$. Moreover, as $ \rho^{2s} g_j \ast_{\omega}  \Phi_{s,\rho} $ converges to $ g_j $ in $ L^2(\He^n)$ (because $\rho^{2s} \Phi_{s,\rho}$ is an approximate identity, see \cite{RTimrn}) we get that 
$$ 
\|  \rho^{2s}  g_j^\lambda \ast_\lambda \Phi_{s,\rho}^\lambda \ast_\lambda \varphi_k^\lambda\|_2^2  \rightarrow \|  g_j^\lambda \ast_\lambda \varphi_k^\lambda\|_2^2.
$$ 
Using Fatou's lemma in \eqref{ineqF} along with these two observations we conclude that
$$  
\sum_{j=0}^\infty  \| g_j\|_2^2 = c  \sum_{j=0}^\infty  \int_{-\infty}^\infty  \Big( \sum_{k=0}^\infty \|    g_j^\lambda  \ast_\lambda \varphi_k^\lambda\|_2^2  \Big) d\mu(\lambda) \leq C. 
$$
This proves that the series 
$ \sum_{j=0}^\infty  g_j(x,u,\xi) Y_j(\omega)$ converges in $  L^2(\He^n, L^2(\mathbb{S}^{m-1}))$ to a function which we denote by $ g^\omega(x,u,\xi)$. We have thus proved $ R_\omega U_\rho = \rho^{2s} g^\omega \ast_{\omega} \Phi_{s,\rho} $. As $ R_\omega U_\rho $ converges to $ g^\omega $ it follows that
$$  
\int_{-\infty}^\infty  \int_{\He^n} e^{-i\lambda \xi} R_\omega U_\rho(x,u,\xi) \,dx \,du\, d\xi \rightarrow  \int_{-\infty}^\infty  \int_{\He^n} e^{-i\lambda \xi} g^\omega(x,u,\xi) \,dx\, du\, d\xi.
$$
This allows us to conclude that (as in the proof of Proposition \ref{prop:coco}) $ g^\omega $ satisfies the compatibility condition \eqref{cocoH} and hence we can get $ f \in L^2(N) $ such that $ R_\omega f = g^\omega$. Consequently, $ R_\omega U_\rho = \rho^{2s}  R_\omega f \ast_{\omega} \Phi_{s,\rho}.$ If we can show that $ (\Phi_{s,\rho}^N)_\omega = \Phi_{s,\rho} $ then we have $ R_\omega U_\rho = \rho^{2s}  R_\omega f \ast_{\omega} (\Phi_{s,\rho}^N)_\omega$, which will prove the theorem by the inversion formula for the Radon transform.

Thus we need to show $ (\Phi_{s,\rho}^N)_\omega = \Phi_{s,\rho} $. The Radon transform of a radial function is radial, and there is a formula to compute the Radon transform in this case, see e.g. \cite{M}. Since $\Phi_{s,\rho}^N(x,u,t)$ is a radial function of $t\in \R^m$, we have, for $m>1$,
\begin{align*}
&\int_{ y \cdot \omega =\xi} ((\rho^2+|x|^2+|u|^2)^2+16|y|^2)^{-\frac{(n+m+s)}{2}}\,dy\\
&\quad = \frac{2\pi^{(m-1)/2}}{\Gamma((m-1)/2)}\int_{|\xi|}^{\infty}((\rho^2+|x|^2+|u|^2)^2+16r^2)^{-\frac{(n+m+s)}{2}}(r^2-\xi^2)^{\frac{m-3}{2}}r\,dr\\
 &\quad =
c_{n,m,s}((\rho^2+|x|^2+|u|^2)^2+16 \xi^2)^{-\frac{(n+1+s)}{2}}=\frac{\pi^{n+(m+1)/2}\Gamma(s)}{2^{2m+n+s-1}\Gamma(\frac{n+1+s}{2})\Gamma(\frac{n+m+s}{2})}\Phi_{s,\rho}(x,u,\xi).
\end{align*}
where $
c_{n,m,s}=\frac{\pi^{\frac{m-1}{2}}}{4^{m-1}}\frac{\Gamma(\frac{n+1+s}{2})}{\Gamma(\frac{n+m+s}{2})}$ (we computed an Abel type integral). In view of \eqref{PhiH}, we get that 
$(\Phi_{s,\rho}^N)_{\omega}(x,u,\xi)=\Phi_{s,\rho}(x,u,\xi)$.

The proof is complete.
\end{proof}

 \subsection{Holomorphic extensions of eigenfunctions on solvable extensions of $H$-type groups} 

Analogously as in the case of the Heisenberg group $\He^n$, we will observe the connection between the solutions of the extension problem \eqref{ep}
and the eigenfunctions of the Laplace--Beltrami operator $\Delta_S$ on the solvable extension $S$ of the $H$-type group $N$, and we will deduce holomorphic properties of these eigenfunctions from Theorem \ref{charHtypecor}.

Let us extend to $H$-type groups the concepts that we introduced in Section \ref{complex}.  Recall that the $H$-type group $N$ admits nonisotropic dilations. Thus there is an action of $A=\R^+$ on $N$. We can therefore form the semidirect product of $N$ and $A$ which is usually denoted by $S= NA$. This group $S$ is solvable and when $N$ is an Iwasawa group coming out of a semisimple Lie group, $S$ can be identified with a noncompact Riemannian symmetric space of rank one. A basis for the Lie algebra $ \mathfrak{s} $ of $S$ is given by $ E_j = \sqrt{\rho}X_j$, $j=1,2,...,2n$, $T_k = \rho Z_k$, $k =1,2,...,m $ and $ H = \rho \partial_\rho$. The Laplace--Beltrami operator $\Delta_S$ (by using the same notation as in Section \ref{complex}) on $S$ is defined by  
\begin{equation*}  
\Delta_S =  \sum_{j=1}^{2n} E_j^2 +\sum_{k=1}^m T_k^2 +H^2 -\frac{1}{2}Q H
\end{equation*} 
where $ Q = 2(n+m) $ is the homogeneous dimension of $ N$. It can be shown that $U$ is a solution of the equation
$$
\big( -\mathcal{L}_N + \partial_\rho^2 +\frac{1-2s}{\rho} \partial_\rho +\frac{1}{4}\rho^2 \Delta_{t} \big) U(x,u,t,\rho) = 0,
$$
 if and only if the function 
$$
\mathbb{W}(x,u,t,\rho)=\rho^{\frac{n+m-s}{2}}U(2^{-1/2}(x,u),2^{-1}t,\sqrt{2\rho})
$$
satisfies the eigenfunction equation
$$ 
- \Delta_S \mathbb{W}(x,u,t,\rho) = \gamma(n+m-\gamma)\mathbb{W}(x,u,t,\rho),\quad \gamma=\frac{1}{2}(n+m+s),\quad s>0.
$$
Moreover we can write, for any $ f \in L^2(N) $,
$$
 \mathbb{W}(x,u,t,\rho) =  2^s \rho^{\frac{n+m+s}{2}} f \ast \Phi_{s, \sqrt{2\rho}}^N(2^{-1/2}(x,u),2^{-1}t). 
  $$
By taking modified Radon transform on both sides of the identity, we have that, for each $\omega\in \mathbb{S}^{m-1}$
$$
R_{\omega} \mathbb{W}(x,u,\xi,\rho)= 2^s \rho^{\frac{n+m+s}{2}}2^{n+m+s}2^{-(n+1+s)}R_{\omega}f \ast_{\omega} \Phi_{s, \sqrt{2\rho}}(2^{-1/2}(x,u),2^{-1}\xi),
$$
since 
\begin{equation}
\label{wPhi}
( \Phi_{s, \sqrt{2\rho}}^N(2^{-1/2}(x,u),2^{-1}t))_{\omega}=2^{m-1} \Phi_{s, \sqrt{2\rho}}(2^{-1/2}(x,u),2^{-1}\xi).
\end{equation}
We will use the notation $\mathbb{W}_{\rho}(x,u,t):=\mathbb{W}(x,u,t,\rho)$. By Theorem \ref{charHtypecor}, the function $ R_{\omega} \mathbb{W}_\rho $ extends to $ \Omega_{\gamma \sqrt{2\rho}} \times \mathbb{S}^{m-1} $ as an $ L^2(\mathbb{S}^{m-1})$-valued 
holomorphic function $ \widetilde{R_{\omega}\mathbb{W}_\rho}(z,w,\zeta)$ and satisfies the uniform estimate 
$$  
\int_{\Omega_{\gamma \sqrt{2\rho}}} \Big( \int_{\mathbb{S}^{m-1}} | \widetilde{R_{\omega}\mathbb{W}_\rho}(z,w,\zeta) |^2 d\sigma(\omega) \Big) \,dz\, dw\, d \zeta \leq C  \rho^{2n+m+1-s}\|f\|_2^2
$$ 
for all $ \rho >0 $. 
 
Let us restate this estimate in terms of  Poisson integrals on the solvable group $ S$. We define
 $$ 
 \varphi_{s,\delta}^N(x,u,t) = \big((\delta+\frac{1}{4}|(x,u)|^2)^2+|t|^2\big)^{-\frac{n+m+s}{2}}
 $$ 
Observe that $ \Phi_{s,\rho}^N $ can be written in terms of $ \varphi_{s,\delta}^N $ as
\begin{equation*} 
\Phi_{s,\rho}^N(x,u,t) =\frac{2^{-(n+1+s)}\Gamma(\frac{n+1+s}{2})\Gamma(\frac{n+m+s}{2})}{\pi^{n+(m+1)/2}\Gamma(s)}\varphi_{s,\delta}^N(x,u,t)
 \end{equation*}
 with $ \delta = \frac{1}{4}\rho^2$. In other words, 
 \begin{equation}
 \label{relaN}
 \Phi_{s,\sqrt{2\rho}}^N(2^{-1/2}x,2^{-1/2}u,2^{-1}t) = \frac{2^{-(1-m)}\Gamma(\frac{n+1+s}{2})\Gamma(\frac{n+m+s}{2})}{\pi^{n+(m+1)/2}\Gamma(s)}   \varphi_{s,\rho}^N(x,u,t).
 \end{equation}
 From \eqref{wPhi}, \eqref{relaN} and \eqref{relH}, we have that
 \begin{align}
 \label{varphis}
 \notag \frac{2^{m-1}\Gamma(\frac{n+1+s}{2})\Gamma(\frac{n+m+s}{2})}{\pi^{n+(m+1)/2}\Gamma(s)}  ( \varphi_{s,\rho}^N(x,u,t))_{\omega}&=2^{m-1} \Phi_{s, \sqrt{2\rho}}(2^{-1/2}(x,u),2^{-1}\xi)\\
  &=2^{m-1}\frac{\Gamma(\frac{n+1+s}{2})^2 }{ \pi^{n+1}\Gamma(s)}   \varphi_{s,\rho}(x,u,\xi).
 \end{align}
 Now we define $ g(x,u,t) :=  f(2^{-1/2}x,2^{-1/2}u,2^{-1}t)$. A calculation shows that
\begin{equation*}
 \mathbb{W}(x,u,t,\rho) = 2^{-(n+m-s)} \rho^{\frac{n+m+s}{2}} \frac{2^{-(1-m)}\Gamma(\frac{n+1+s}{2})\Gamma(\frac{n+m+s}{2})}{\pi^{n+(m+1)/2}\Gamma(s)}  g\ast \varphi_{s,\rho}^N(x,u,t).
\end{equation*}

The kernels $\varphi_{i\lambda,\rho}^N(x,u,t)$ defined for $\lambda\in \C$ are the generalised Poisson kernels and they occur in the definition of the Helgason Fourier transform on the solvable group $S$. This suggests that the right Poisson transform for the solvable extension $S$ of $N$ is given by
 $$ 
 \mathcal{P}_\lambda^N f(x,u,t,\rho)  =\rho^{\frac{n+m-i\lambda}{2}}  \big( \varphi_{i\lambda,1}^N(\cdot)f(\cdot)\big)\ast \varphi_{-i\lambda,\rho}^N(x,u,t) 
 $$
 where $ f $ is a function on $N$. 

In view of the description above, from Theorem \ref{charHtypecor} we can deduce as a corollary the following characterisation of certain eigenfunctions of $ \Delta_S $ expressible as Poisson transforms of functions on $N$.

  \begin{thm}
  \label{charEH}  
  Let $ 0 < s < 1$ and $m>1$. An eigenfunction $ \mathbb{W}$ of the Laplace-Beltrami operator $ \Delta_S $ on $ S $ with eigenvalue $ -\frac{1}{4}((n+m)^2-s^2) $  can be expressed as the Poisson integral $ \mathcal{P}_{is}^Nf $ with $ f \in L^2(N, (\varphi_{0,1}^N(h))^2 dh) $ if and only if  there exists a $ \gamma> 0 $ such that the function $ (x,u,\xi,\omega) \rightarrow R_{\omega}\mathbb{W}_\rho(x,u,\xi) $ extends to $ \Omega_{\gamma \sqrt{\rho}} \times \mathbb{S}^{m-1} $ as an $ L^2(\mathbb{S}^{m-1})$-valued 
holomorphic function $ \widetilde{R_{\omega}\mathbb{W}_\rho}(z,w,\zeta)$ and satisfies the uniform estimate 
$$  
\int_{\Omega_{\gamma \sqrt{\rho}}} \Big( \int_{\mathbb{S}^{m-1}} |  \widetilde{R_{\omega}\mathbb{W}_\rho}(z,w,\zeta) |^2 d\sigma(\omega) \Big) \,dz\, dw\, d \zeta \leq C   \rho^{2n+m+1-s}
$$ 
for all $ \rho >0 $.  
  \end{thm}
  \begin{proof} 
  When $ f \in L^2(N, (\varphi_{0,1}^N(h))^2 dh) $ we know that $ R_\omega  ( \varphi_{i\lambda,1}^N f) \in L^2(\He^n) $ and so the function $ g_\lambda^\omega(h) := \varphi_{i\lambda,1}(h)^{-1} R_\omega  (\varphi_{i\lambda,1}^N f)(h)  \in L^2(\He^n, \varphi_{0,1}(h)^2 dh)$.  Therefore, in view of \eqref{Poisson} and \eqref{varphis} $R_\omega( \mathcal{P}_\lambda^N f)(h,\rho) =  c_{n,m,s}\rho^{(m-1)/2} \mathcal{P}_\lambda (g_\lambda^\omega) (h,\rho)$. Hence the theorem follows from the corresponding result for the Heisenberg group, namely Theorem~\ref{charE}.
  \end{proof}
 
\subsection*{Acknowledgments}

L. Roncal is supported by the Basque Government through the BERC 2018-2021 program, by the Spanish Ministry of Economy and Competitiveness MINECO: BCAM Severo Ochoa excellence accreditation SEV-2017-2018  and through project MTM2017-82160-C2-1-P funded by (AEI/FEDER, UE) and acronym ``HAQMEC''. She also acknowledges the RyC project RYC2018-025477-I and IKERBASQUE. S. Thangavelu is supported by  J. C. Bose Fellowship from D. S. T., Government of India.


\end{document}